\documentclass[a4paper, 10pt]{amsart}
\usepackage[left=3cm, right=3cm, top=4cm, bottom=4cm]{geometry}
\usepackage{amsmath, amssymb, amsthm, textcomp, verbatim, graphicx, xcolor, mathtools}
\usepackage{wrapfig, comment}
\usepackage{xparse}
\usepackage{pifont}
\usepackage[utf8]{inputenc}
\usepackage[T1]{fontenc}
\usepackage[all]{xy}
\usepackage{hyperref}
\hypersetup{colorlinks = true, linkcolor = blue, filecolor = blue, urlcolor = blue, citecolor=blue}
\theoremstyle{definition} \newtheorem{defn}{Definition}[section]
\theoremstyle{remark} \newtheorem{rem}[defn]{Remark}
\theoremstyle{remark} \newtheorem*{rem*}{Remark}
\theoremstyle{plain} \newtheorem{thm}[defn]{Theorem}
\theoremstyle{plain} \newtheorem{lem}[defn]{Lemma}
\theoremstyle{plain} \newtheorem{prop}[defn]{Proposition}
\theoremstyle{plain} 
\theoremstyle{remark} 
\theoremstyle{remark} 
\theoremstyle{remark} 

\numberwithin{equation}{section} 

\newcommand{\mc}{\mathcal}

\newcommand{\x}{\times}

\newcommand{\R}{\mathbb{R}}

\newcommand{\N}{\mathbb{N}}

\newcommand{\pr}{\mathbb{P}}
\newcommand{\E}{\mathbb{E}}

\newcommand{\Exp}{\operatorname{Exp}}

\newcommand{\ind}{\mathbf{1}}

\begin{document}		
	\title{Weak Uniqueness for the Stochastic Heat Equation Driven by a Multiplicative Stable Noise}	
	\author{Sayantan Maitra}
	\thanks{Indian Statistical Institute, Bangalore Centre, 8th Mile Mysore Road, Bengaluru - 560059, Karnakata. \\ E-mail: \href{mailto:sayantanmaitra123@gmail.com}{sayantanmaitra123@gmail.com} }

	\maketitle
	
	\markright{WEAK UNIQUENESS FOR THE STOCHASTIC HEAT EQUATION}
	
	\begin{abstract}
	We consider the stochastic heat equation 
	\begin{align}\label{eq:main_abs} 
		\frac{\partial Y_t(x)}{\partial t}  = \frac{1}{2} \Delta Y_t(x) + Y_{t-}(x)^{\beta} \dot{L}^{\alpha}_{x,t},  \tag{$\star$}
	\end{align}
	with $t \ge 0$, $x \in \R$ and $L^{\alpha}$ being an $\alpha$-stable white noise without negative jumps. Under appropriate non-negative initial conditions, when $\alpha \in (1,2)$ and $\beta \in (\frac{1}{\alpha},  1)$ we prove that weak uniqueness holds for \eqref{eq:main_abs} using the approximating duality approach developed by Mytnik \cite{mytnik98}.\\		
	\end{abstract}

	\section{Introduction}
	
	In this paper we study the stochastic equation 
	\begin{align}\label{eq:main_spde}
		\frac{\partial Y_t(x)}{\partial t} & = \frac{1}{2} \Delta Y_t(x) + Y_{t-}(x)^{\beta} \dot{L}^{\alpha}_{x,t}, \qquad t \geq 0, x \in \R		
	\end{align} where $L^{\alpha}$ is a stable noise of index $\alpha$ without negative jumps and $\Delta = \frac{\partial ^2}{\partial x^2}$. We show that its solutions are unique in law when $1<\alpha<2$, $0<\beta <1$ and $1<\alpha\beta$ (see Theorem \ref{thm:main_spde}). Proving uniqueness in law, also called \textit{weak uniqueness}, is an important step for establishing that a model of an underlying system of interacting particles converges to a limit. 
	
	When $\alpha=2$ the above equation is similar to the following SPDE
	\begin{align}\label{eq:density_SBM}
		\frac{\partial Y_t(x)}{\partial t} = \frac{1}{2} \Delta Y_t(x) + Y_{t}(x)^{\beta} \dot{W}_{x,t}
	\end{align}
	where $\dot{W}$ is the Gaussian space-time white noise. When $\beta = \frac{1}{2}$ this describes the density process of the super-Brownian motion (SBM) in $\R$ which can be obtained as the scaling limit of interacting branching Brownian motions. The weak uniqueness of \eqref{eq:density_SBM} with $\beta = \frac{1}{2}$ follows from the martingale problem formulation of SBM using duality (see \cite[Theorem II.5.1]{perkins02}). In the $\beta \in (\frac{1}{2},1)$ case, the weak existence of \eqref{eq:density_SBM} was proved by Mueller, Perkins \cite{mp92} and the weak uniqueness was established by Mytnik \cite{mytnik98}. The question of pathwise uniqueness of \eqref{eq:density_SBM} for $\beta > \frac{3}{4}$ was settled by Mytnik and Perkins in 2011 \cite{myt-per11}. Some negative results are also known: \cite{mmp14} proved pathwise non-uniqueness of solutions to \eqref{eq:density_SBM} for $\beta \in (0, \frac{3}{4})$ and \cite{bmp10} showed pathwise non-uniqueness for $\beta \in (0, \frac{1}{2})$ with an added non-trivial drift. 		  
	
	Let us now consider \eqref{eq:main_spde} where $x \in \R^d$.  For $\alpha \neq 2$, Mueller \cite{mueller98} proved a certain short time (strong) existence of solution to (\ref{eq:main_spde}) under the relations $d < \frac{2(1-\alpha)}{\alpha\beta - (1-\alpha)}, \alpha \in (0,1)$. The weak existence was shown by \cite{mytnik02} under the relations $0< \alpha \beta < \frac{2}{d}+1$, $1 < \alpha < \min(2, \frac{2}{d}+1)$. When $d=1$ and $\alpha\beta =1$ it is known that \eqref{eq:main_spde} describes the density of the super-Brownian motion with $\alpha$-stable branching mechanism; see \cite{myt-per06} for more details. The weak uniqueness for this case was proved in \cite{mytnik02} while the same for the general case was left open (see \cite[Remark 5.9]{mytnik02}). As stated earlier in our main result we resolve the question for the case $d=1$, $1<\alpha<2$, $0<\beta <1$ and $1<\alpha\beta$. 
	
	It is known that pathwise uniqueness implies weak uniqueness for \eqref{eq:main_spde}. But as the coefficient of the noise term in  \eqref{eq:main_spde} is not Lipschitz, standard techniques such as Gr\"{o}nwall's inequality cannot be used to prove pathwise uniqueness of solutions. However, more recently Yang and Zhou \cite{yz17} have established pathwise uniqueness for \eqref{eq:main_spde} in the regime $\frac{2(\alpha -1)}{(2-\alpha)^2} < \beta < \frac{1}{\alpha}+\frac{\alpha -1}{2}$ and $d=1$. This region partially, although not fully, overlaps with the one stated in Theorem \ref{thm:main_spde}. See the Figure \ref{fig:parameter_space} for an illustration of the $(\alpha, \beta)$ parameter space. 
	
	\begin{figure}[t]\label{fig:parameter_space}
		\includegraphics[width=8cm]{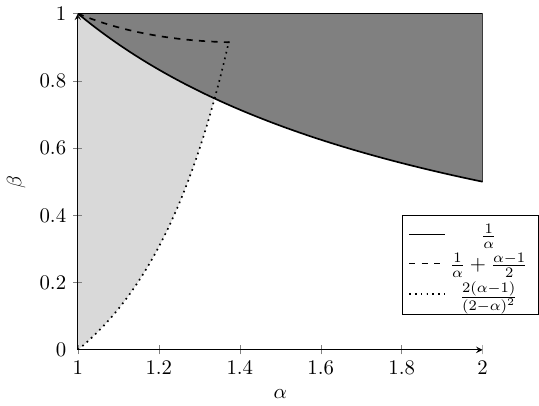}
		\centering
		\caption{Mytnik proved weak uniqueness for of \eqref{eq:main_spde} when $\beta = \frac{1}{\alpha}$. Yang, Zhou showed pathwise uniqueness of this equation when $(\alpha, \beta)$ falls in the light gray region. Our main result, Theorem \ref{thm:main_spde}, proves weak uniqueness for \eqref{eq:main_spde} when $(\alpha,\beta)$ is in the dark gray region.}
		\centering
	\end{figure}
	
	In the next subsection we precisely define our model and state the main theorem.
	
	\subsection{Model and Main Result}
	
	To define our model and state the main result we need to introduce the following notations. Let $\lVert f \rVert_{\infty} = \sup_{x \in \R} |f(x)|$ and  $\lVert f\rVert_p = \left(\int_{\R} |f(x)|^p \, d x\right)^{1/p}$ for $p \ge 1$ be the norms of the spaces ${\bf{L}}^{\infty}(\R)$ and ${\bf{L}}^p(\R)$ respectively. The norms on ${\bf{L}}^{\infty}([0,T]\x R)$ and ${\bf{L}}^p ([0,T]\x \R)$ are defined similarly. By ${\bf{L}}^p_{loc}(\R_+ \x \R)$ we will mean the collection of measurable functions $f:\R_+ \x \R \to \R$ such that $\int_0^T \int_{\R} |f(s,x)|^p \, d s \, dx < \infty$ for all $T \in (0, \infty)$. We also define $\mc{S}\equiv\mc{S}(\R)$ to be the space of all smooth rapidly-decreasing functions defined on $\R$ whose derivatives of all orders are also rapidly-decreasing. The subsets of ${\bf{L}}^p(\R)$ and $\mc{S}(\R)$ containing all non-negative functions are denoted by ${\bf{L}}^p(\R)_+$ and $\mc{S}_+ \equiv \mc{S}(\R)_+$ respectively. 
	
	Let $M_F \equiv M_F(\R)$ be the set of all non-negative finite measures on the real line, $\R$, with the topology of weak convergence. We denote the space of all c\'{a}dl\'{a}g paths in $M_F$ as $D \equiv D ([0, \infty), M_F)$. This space is equipped with the topology of weak convergence and $\mc{B}(D)$ denotes the Borel $\sigma$-algebra on $D$. Similarly, $\mc{B}(\R)$ denotes the $\sigma$-algebra of all Borel measurable subsets of $\R$ and for $E \in \mc{B}(\R)$ we  use $|E|$ for the Lebesgue measure of $A$.  For $\mu \in M_F$ and $\varphi \in \mc{S}$ we denote $\langle \mu, \varphi \rangle  = \int_{\R} \varphi \, d \mu$. We will often identify a measurable function $f: (\R, \mc{B}(\R)) \to \R$ with $f(x) \, dx$, where $dx$ denotes the Lebesgue measure. In this case, $\langle f, \varphi \rangle : = \int_{\R} f(x) \varphi(x)\,  dx$.
	
	\begin{defn}\label{def:stable_mart_measure}
		Let $\alpha \in (0,2)$. Suppose for each $E \in \mc{B}(\R)$ with $|E| < \infty$, $\{L^{\alpha}_t(E)\}_{t \geq 0}$ is a martingale defined on some filtered probability space $(\Omega, \mc{F}, (\mc{F}_t)_{t \ge 0}, \pr)$ and 
		\begin{align}\label{eq:stable_mart_measure_def}
			\E \exp{(-\lambda L^{\alpha}_t(E))} = \exp{(\lambda^{\alpha} t |E|)}, 
		\end{align}
		for all $t\ge 0$, $\lambda \ge 0$. Then we call $L^{\alpha}$ an \textit{$\alpha$-stable martingale measure on $[0,\infty) \times \R$ without negative jumps}.
	\end{defn}
	
	Observe that $L^{\alpha} (E\times [0,t]) : = L^{\alpha}_t (E)$ is indeed a martingale measure in the sense of Walsh \cite[Chapter 2]{walsh86}.
	
	\begin{defn}\label{def:weak_form}
		Let $Y_0 \in M_F$. Given an $\alpha$-stable martingale measure $L^{\alpha}$ on $(\Omega, \mc{F},(\mc{F}_t)_{t \ge 0}, \pr)$ without negative jumps, a two-parameter stochastic process $\{Y_t (x)\}_{t\ge 0, x \in \R}$ defined on the same probability space is said to solve \eqref{eq:main_spde} if the following hold.
		\begin{itemize}
			\item[(i)] $Y$ is adapted to $(\mc{F}_t)_{t\ge 0}$.
			
			\item[(ii)] The map $t \mapsto Y_{t}(x) \, dx$ defines an $M_F$-valued c\'{a}dl\'{a}g process. In other words, $Y \in D ([0, \infty), M_F)$ a.s..
			
			\item[(iii)] For all $\psi \in \mc{S}(\R)$  and $t \geq 0$,
			\begin{align}\label{eq:main_spde_weak_form}
				\langle  Y_t, \psi \rangle  =  \langle  Y_0, \psi  \rangle  + \int_0^t \langle  Y_s, \frac{1}{2}\Delta \psi  \rangle \,d s  +  \int_{s \in [0,t]} \int_{x\in \R} (Y_{s-}(x))^{\beta} \psi(x)  L^{\alpha} (d x, \, d s), 			
			\end{align} 	
		\end{itemize}
	\end{defn}	
	Recall that \cite[Theorem 1.5]{mytnik02} guarantees the weak existence of such a solution $Y \equiv (Y_t)_{t \ge 0}$. Also, it was shown in \cite[Proposition 4.1]{mytnik02} that 
	\begin{align*}
		Y \in D([0, \infty), M_F) \cap {\bf{L}}^{\rho}_{loc}(\R_+ \x \R)  \quad  (1<\rho<3).
	\end{align*}  
	
	We need to recall some notions of existence and uniqueness for solutions to \eqref{eq:main_spde} that will be used in this article.  	
	\begin{itemize}
		\item (\ref{eq:main_spde}) is said to admit a \textit{weak solution} with initial condition $Y_0$ if there exists a filtered probability space $(\Omega, \mc{F}, \{\mc{F}_t\}_{t\ge 0}, \pr)$ and an $\{\mc{F}_t\}$-adapted pair $(Y,L^{\alpha})$ such that $L^{\alpha}$ satisfies (\ref{eq:stable_mart_measure_def}) and (\ref{eq:main_spde_weak_form}) holds. 
		
		
		\item \textit{Weak uniqueness} holds for (\ref{eq:main_spde}) if whenever the pairs $(Y, L^{\alpha})$ and $(\tilde{Y}, \tilde{L}^{\alpha})$ satisfy (\ref{eq:main_spde_weak_form}) with the same initial condition, they have the same finite dimensional distributions.
		
		
	\end{itemize}
	
	Our main result is the following.	
	
	\begin{thm}\label{thm:main_spde}
		Assume that $1<\alpha<2$ and $\frac{1}{\alpha}<\beta < 1$ and $Y_0 \in M_F$. Then weak uniqueness holds for solutions to \eqref{eq:main_spde_weak_form}, i.e. if $(Y,L^{\alpha})$ and $(\tilde{Y}, L^{\alpha})$ are both weak solutions of \eqref{eq:main_spde_weak_form} and $Y_0 = \tilde{Y}_0$, then $Y$ and $\tilde{Y}$ have the same finite dimensional distributions.
	\end{thm}
	
	We will now describe our approach for proving this result.
	
	\subsection{Proof Strategy}
	
	An approach for showing weak uniqueness of stochastic equations is the following. First, one shows that solutions to \eqref{eq:main_spde_weak_form} are equivalently also solutions of an appropriate martingale problem. Then it is enough to show that any two solutions to the martingale problem have the same one-dimensional distributions (cf. \cite[Theorem 4.4.2]{ek}). We may define the (local) martingale problem as follows. For $\psi \in \mc{S}_+$ and $t \ge 0$ let  	
	\begin{align}\label{eq:main_mart_problem}
		M^Y_t(\psi)  = e^{-\langle Y_t, \psi\rangle} - e^{-\langle Y_0, \psi\rangle}  - \int_0^t  e^{- \langle Y_{s-}, \psi\rangle } \left( -\langle Y_{s-}, \frac{1}{2} \Delta \psi\rangle + \langle Y_{s-}^{\alpha \beta}, \psi^{\alpha} \rangle\right)  \,d s
	\end{align} 
	where $(Y_t)_{t\ge 0}$ are the coordinate maps on $D$, i.e. $Y_t (\omega) = \omega (t)$ for $\omega \in D$. A probability measure $\pr$ on $(D, \mc{B}(D))$ is said to be a solution of the (local) martingale problem for \eqref{eq:main_mart_problem} if for all $\psi \in \mc{S}_+$ we have that $\{M^Y_t(\psi)\}_t$ is a (local) martingale under $\pr$. We know from \cite[Proposition 4.1]{mytnik02} that a solution to the local martingale problem \eqref{eq:main_mart_problem} exists with stopping times 
	\begin{align}\label{eq:gamma^Y}
		\gamma^Y(k)  := \inf\left\{ s\ge 0 \middle| \int_0^s \lVert Y_{r}\rVert^{\alpha\beta}_{\alpha\beta} \,d r >k \right\}, \quad k \in \N.
	\end{align}
	where $\lVert \cdot \rVert_{\alpha\beta}$ denotes the norm of the space $\bf{L}^{\alpha\beta}(\R)$.  \cite[Proposition 4.1]{mytnik02} also guarantees that, when $Y_0 \in M_F$, for all $t>0$ we have $Y_t \in  M_F$ as well.
	
	Since $Y_0 \in M_F$ is chosen arbitrarily, in light of the above discussion we can rephrase Theorem \ref{thm:main_spde} into the equivalent result concerning the martingale problem \eqref{eq:main_mart_problem}.
	
	\begin{thm}\label{thm:main_mart_problem}
		Under the assumptions of Theorem \ref{thm:main_spde}, any two solutions of \eqref{eq:main_mart_problem} have same one-dimensional distributions. 
	\end{thm}	
	
	Motivated by \cite{mytnik98} we use an \textit{approximating duality} argument and would like to show the following.
	\begin{thm}\label{thm:required_approx_duality_one_sided}
		Let $ \psi \in \mc{S}_+$ and $Y$ be a solution to \eqref{eq:main_mart_problem} where $Y_0$ is as in Theorem \ref{thm:main_spde}. Then there exists a sequence of processes $\{Z^{(n)}\}_{n\ge 1}$, independent of $Y$, with $Z^{(n)}_0 = \psi$ for all $n \geq 1$ such that for each $t \geq 0$
		\begin{align}\label{eq:required_approx_duality_one_sided}
			\E \exp(-\langle Y_t, \psi\rangle) = \lim_{n\to \infty} \E\exp(-\langle Y_0, Z^{(n)}_t\rangle).
		\end{align}		
	\end{thm}	
	By the virtue of \cite[Theorem 1.3]{mytnik96} this will imply Theorem \ref{thm:main_mart_problem}.
	
	Now we shall discuss a formal strategy of how one would prove Theorem \ref{thm:required_approx_duality_one_sided}. Suppose $Y$ solves \eqref{eq:main_mart_problem} and $Z$ solves SPDE given below
	\begin{align}\label{eq:dual_spde}
		\frac{\partial Z_t(x)}{\partial t}  = \frac{1}{2} \Delta_x Z_t(x) + Z_{t-}(x)^{\frac{1}{\beta}} \dot{L}^{\alpha\beta}, \qquad Z_0 = \psi
	\end{align}
	or equivalently the local martingale problem
	\begin{align}\label{eq:dual_mart_problem}
		M^Z_t(\varphi)  = e^{-\langle  \varphi, Z_t\rangle} - e^{-\langle \varphi, Z_0\rangle}  - \int_0^{t} e^{- \langle \varphi, Z_{s-}\rangle } \left( -\langle  \frac{1}{2} \Delta \varphi, Z_{s-}\rangle + \langle \varphi^{\alpha\beta} ,  Z_{s-}^{\alpha},\rangle\right) \,d s, \quad \varphi \in \mc{S}_+
	\end{align} 
	is an $\mc{F}^{Z}_{t}$-local martingale. Then one could try to establish the following exponential duality relation
	\begin{equation}\label{eq:required_duality_relation}
		\E\exp(-\langle Y_t, \psi\rangle) = \E \exp(-\langle \varphi, Z_t\rangle).
	\end{equation}
	Although this duality relationship holds, the required integrability conditions will fail to hold (see \cite[Theorem 4.4.11]{ek}). Thus one uses the approximate duality technique. In this approach we construct an approximating sequence $Z^{(n)}$ to $Z$ using the framework of \cite{mytnik98} to prove the Theorem \ref{thm:required_approx_duality_one_sided}.
	
	However there are two key difficulties to overcome. First, we require suitable bounds on the moments of the solutions to \eqref{eq:main_mart_problem} and the second difficulty is the fact that $M^Y(\psi)$, as defined above, are only local martingales. We prove the moment estimate result in Proposition \ref{prop:moment_est_Y_2} for the range of $\alpha$ and $\beta$ stated in Theorem \ref{thm:main_spde}. From this we can show that $M^Y(\psi)$ is indeed a martingale.	

	\begin{rem}
		Note that the condition $\alpha\beta >1$ is crucial for our argument and the technique of approximate duality. Consequently the case when $\alpha \beta < 1$ is not covered by this method. 
	\end{rem}
	
	\textbf{Layout of the paper.} We briefly sketch the construction of $Z^{(n)}$ in the next section. The moment estimates and the proof that $M^Y(\psi)$ is a martingale can be found in Section \ref{sec:moment_mart}. We have split the proof of Theorem \ref{thm:required_approx_duality_one_sided} into Propositions \ref{prop:lem3.1_m98}, \ref{prop:(3.19)_m98} and \ref{prop:lem3.3_m98} and have stated them in Section \ref{sec:proof_thm}. In this section we also finish the proof of Theorem \ref{thm:required_approx_duality_one_sided} assuming these three results. Their proofs can be found in Section \ref{sec:proof_props}. Finally, Appendix \ref{app:gronwall}, \ref{app:pde} and \ref{app:time_dep_mart} contain some auxiliary results that are used in various places of this paper. 
	
	We use the notations $c$, $c_1$, $C$, $C_1$ etc. to denote constants whose value may change from one line to the next. They will usually depend on the time horizon $T$ and the initial condition $Y_0$. Wherever necessary we will denote their dependence on the relevant parameters.
	
	\textbf{Acknowledgments:} This paper is part of my Ph.D. thesis. I would like to thank my advisor Siva Athreya for proposing this problem to me and for numerous helpful and motivating conversations. I want to specially thank Leonid Mytnik for several useful discussions and pointers on techniques used in the proof including the moment estimate derived in Proposition \ref{prop:moment_est_Y_2}. I am grateful to Yogeshwaran D, Edwin Perkins and B V Rao for their useful suggestions and comments on an earlier draft of this paper. Lastly, I wish to thank the two anonymous reviewers for carefully reading the draft and pointing out various mistakes.
	
	\section{Preliminaries}\label{sec:prelim}
	
	This section contains notations that are used throughout the paper, some useful results regarding the mild forms of \eqref{eq:main_spde_weak_form} and the construction of the approximating sequence $Z^{(n)}$.
	
	Let $p_t(x) = \frac{1}{\sqrt{2\pi t}} \exp\left( - \frac{x^2}{2 t}\right)$ for all $t > 0, x \in \R$. For any function $f:\R\to \R$ and a measure $\mu \in M_F$, we will denote 
	\begin{align*}
		P_t f(x) = \int_{\R} p_t(x-y)f(y) \, dy\text{ and } P_t \mu (x)  = \int_{\R} p_t(x-y)\, \mu( dy).
	\end{align*}	
	As in \cite{yz17}, define a measure on $\R$, 
	\begin{align}\label{eq:m_0}
		m_0(d z) = \frac{\alpha(\alpha-1)}{\Gamma(2-\alpha)}z^{-1-\alpha} \ind\{z>0\}\,d z.
	\end{align}
	We first show that the solution $Y_t$ in \eqref{eq:main_spde_weak_form} of Definition \ref{def:weak_form} can be written in the following equivalent mild forms.  
	
	\begin{prop}
		Let $Y$ be a solution as in Definition \ref{def:weak_form} and $Y_0$ be as in Theorem \ref{thm:main_spde}. Then
		\begin{itemize}
			\item[(a)] For $t\ge 0, x\in \R$,
			\begin{align}\label{eq:main_spde_mild}
				Y_t(x) =P_t Y_0(x) + \int_0^t \int_{\R} p_{t-s}(x-y)Y_s(y)^{\beta} \, L^{\alpha}(dy, \, ds).
			\end{align}
			
			\item[(b)] There exists a Poisson random measure (PRM) $N$ on $(0,\infty)^2\x \R$ with intensity $ds\, m_0(dz) dx$ such that 
			\begin{align}\label{eq:main_spde_mild_PRM_1}
				Y_t(x) =P_t Y_0(x)  + \int_{0}^t \int_{0}^{\infty}\int_{\R} z \,p_{t-s}(x-y)Y_s(y)^{\beta}  \, \tilde{N}(dy, \, dz, \, ds),
			\end{align}
			where $\tilde{N} (dy, \, dz, \, ds) = N (dy, \, dz, \, ds) - dy \, m_0(dz)\, ds$.
			
			\item[(c)] On an enlarged probability space there exists a PRM $N_0$ on $(0,\infty)^2\x \R\x (0,\infty)$ with intensity $ds\, m_0(dz)\, dy\, dv$ such that, for all $t\ge 0$ and a.e. $x \in \R$, 
			\begin{align}\label{eq:main_spde_mild_PRM_2}
				Y_t(x) =P_t Y_0(x) + \int_{0}^t \int_{0}^{\infty}\int_{\R}\int_{0}^{Y_s(y)^{\alpha\beta}} z\, p_{t-s}(x-y)\tilde{N}_0(\, d v, \, dy, \, d z, \, d s),
			\end{align}
			where $\tilde{N}_0(\, d v, \, dy, \, d z, \, d s) = N_0 (\, d v, \, dy, \, d z, \, d s)-  \, dv \, dy \, m_0(dz) \, ds $.
		\end{itemize}		
	\end{prop} 
	
	\begin{proof}
		(a) This can be shown by an argument similar to the one in the proof of Theorem 1.1(a) in \cite{myt-per06}. The only difference here is to show that for each $t>0$,
		\begin{align}
			\int_0^t \int_{\R} (t-s)^{-\alpha/2} Y_s(y)^{\alpha\beta} \, dy \,ds \le\left( \sup_{s\le t} \lVert Y_s\rVert_{\alpha\beta}^{\alpha\beta}\right) \int_0^t (t-s)^{-\alpha/2}\,d s<\infty \text{ a.s.}
		\end{align}
		This follows from the facts that $\int_0^t \lVert Y_s\rVert^{\alpha\beta}_{\alpha\beta} \,d s<\infty$ and $s\mapsto \lVert Y_s\rVert_{\alpha\beta} = \langle Y_s^{\alpha\beta}, 1\rangle^{\frac{1}{\alpha\beta}}$ is a cadlag map. 
		
		The claim in part (b) follows from the above and \cite[Theorem 1.1(a)]{myt-per06}. Using a change of variable type transformation as indicated in the proof of \cite[Proposition 2.1]{yz17} we get part (c). 
	\end{proof}
	
	\begin{rem}
		To show that the stochastic integral in \eqref{eq:main_spde_mild_PRM_2} is well-defined Yang and Zhou used the additional condition (see \cite[Assumption 1.4]{yz17}) that there is a $q>\frac{3\alpha\beta}{3-\alpha}$ such that $\int_0^t  \int_{\R}  Y_s(x)^q \, dx \, d s< \infty$ for all $t>0$ a.s.. We here observe that our proof of \eqref{eq:main_spde_mild_PRM_2} above does not require this assumption.
	\end{rem}
	
	As mentioned before we need to construct an approximating sequence $\{Z^{(n)}\}_n$ to $Z$ described in \eqref{eq:dual_spde}. We shall use the construction given in \cite[\textsection 3]{mytnik02}. For completeness we only present the sketch below.
	
	Define $Z^{(n)}_0 (\,d x) = \psi(x)\,d x$ and let $b_n = \frac{\alpha \beta}{\Gamma(2-\alpha\beta)}n^{\alpha \beta -1}$. We know from \cite[Proposition A2]{fleisch86} that given $\mu \in M_F$, there is a unique non-negative solution to the partial differential equation (PDE) 
	\begin{align}\label{eq:pde_Z}
		v_t = P_t \mu - \int_0^t P_{t-s}(b_n v_s^{\alpha}) \,d s
	\end{align}
	where $(P_t \mu) (x)= \int_{\R} p_t(x-y) \mu(\,d y)$. Let us call this solution $V^{n}_{\cdot}(\mu)$. See Appendix \ref{app:pde} for some properties of the above PDE under nicer initial conditions.
	
	The idea behind this $Z^{(n)}$ is as follows. $Z^{(n)}$ evolves according to the PDE \eqref{eq:pde_Z}, jumps after a random time given by dirac measures at specified mass and location (denoted in the following by $\gamma^{n}(T^{n}_k)$, $S^n_i$ and $U^n_i$ respectively, see \eqref{eq:Zn_jump_time_1} for precise definition). More precisely, let $\tilde{T}^{Z,n}_i := \tilde{T}^n_i \sim \Exp(n^{\alpha\beta}\frac{(\alpha\beta -1)}{\Gamma(2-\alpha\beta)})$,  $i \in \N$, be i.i.d. random variables and $ T^{n}_i := \sum_{k=1}^i \tilde{T}^{n}_k$. The jump heights are given by i.i.d. $[\frac{1}{n}, \infty)$-valued random variables $ \{S^{n}_i \mid i \in \N\}$ defined by
	\begin{align}
		\pr (S^{n}_i \ge b) = \frac{\int_{b\vee (1/n)}^{\infty} \lambda^{-\alpha\beta -1}\,d \lambda}{\int_{1/n}^{\infty} \lambda^{-\alpha\beta -1}\,d \lambda}, \qquad b \ge 0.	
	\end{align}  
	We observe that $\E[S^n_i] = \frac{\alpha\beta}{n (\alpha\beta -1)}$. Let 
	\begin{align*}
		A^{n}_t := \sum_{k=1}^{\infty} S^{n}_k \ind (T^{n}_k \le t)
	\end{align*} be the process that jumps by height $S^{n}_i$ at time $T^{n}_i$ for all $i \in \N$. By $(\mc{F}^{A^n}_t)_{t\ge 0}$ we will denote the filtration generated by $A^n$. For $0 \le t \le T^{n}_1$ define the time change
	\begin{align}
		\gamma^{n} (t)= \inf \left\{s\geq 0\middle|\int_{0+}^s \lVert V^{n}_r (\mu)\rVert^{\alpha}_{\alpha}\,d r>t \right\}.
	\end{align} 
	We can define the approximating sequence $Z^{(n)}$ on the (random) interval $[0, \gamma^n(T^n_1))$ by 
	\begin{align}
		Z^{(n)}_t = V^{n}_t (Z^{(n)}_0), \quad 0 \leq t < \gamma^{n} (T^{n}_{1}), 
	\end{align}
	where $V^n$ is the solution of the PDE \eqref{eq:pde_Z}. For defining $Z^{(n)}$ at the time $t =  \gamma^{n} (T^{n}_{1})$, we proceed as follows. For each $f \in \bf{L}^{\alpha}(\R)_+$ let $G(f, \cdot)$ be a probability measure on $(\R, \mc{B}(\R))$ such that for all $E\in \mc{B}(\R)$,  
	\begin{align*}
		G(f,E) := \frac{\int_E f(x)^{\alpha}\,d x}{\lVert f \rVert_{\alpha}^{\alpha}}.
	\end{align*}		
	Lastly, let $U^n_1$ be a $\R$-valued random variable defined by the relation 
	\begin{align*}
		\pr(U^{n}_1 \in E \mid \mc{F}^{A^n}_{T^n_1}) = G(Z^{(n)}_{\gamma^n (T^{n}_1)-}, E) , \quad E \in \mc{B}(\R).
	\end{align*} 	
	Then we can define
	\begin{align}\label{eq:Zn_jump_time_1}
		Z^{(n)}_{\gamma^{n} (T^{n}_{1})} = Z^{(n)}_{\gamma^{n} (T^{n}_1)-} + S^{n}_1\delta_{U^{n}_1}.
	\end{align}
	Thus we have constructed $Z^{(n)}$ on the interval $[0, \gamma^{n} (T^{n}_{1})]$.	 
	
	When $t > \gamma^n(T^n_1)$, $Z^{(n)}$ is defined inductively: for integers $k \ge 1$,
	\begin{align}
		Z^{(n)}_t :=  \begin{cases}
			V^{n}_{t-\gamma^{n}(T^{n}_k)}(Z^{(n)}_{\gamma^{n}(T^{n}_k)}),\qquad t \in [\gamma^{n}(T^{n}_k), \gamma^{n}(T^{n}_{k+1})),
			\\
			Z^{(n)}_{\gamma^{n} (T^{n}_k)-} + S^{n}_{k+1}\delta_{U^{n}_{k+1}}, \qquad t = \gamma^{n} (T^{n}_{k+1}),
		\end{cases} 
	\end{align}	
	where
	\begin{align}
		\gamma^{n}(t) = \inf \left\{s \ge 0 \middle| T^{n}_k + \int_{0+}^{s-\gamma^{n}(T^{n}_k)} \lVert V_r^{n} (Z^{(n)}_{T^{n}_k}) \rVert_{\alpha}^{\alpha} \,d r >t \right\}, \qquad T^{n}_k \le t< T^{n}_{k+1}, 
	\end{align}  and 
	\begin{align}
		\pr(U^{n}_{k+1} \in E \mid \mc{F}^{A^n}_{T^n_{k+1}}) := G(Z^{(n)}_{\gamma^n (T^{n}_{k+1})-}, E) , \qquad E \in \mc{B}(\R).
	\end{align}
	
	This completes the construction of $Z^{(n)}$. It is known that $Z^{(n)}$ solves a local martingale problem as described by the following lemma. As usual, $\mc{F}^{Z^{(n)}}$ denotes the filtration generated by $Z^{(n)}$.
	\begin{lem} \label{lem:lem3.7_myt02}
		Let  $ \eta :=  \frac{\alpha \beta(\alpha \beta -1)}{\Gamma (2-\alpha \beta)}$ and 
		\begin{align}\label{eq:incomplete_gamma}
			g(r,y) : = \int_{0+}^r (e^{-\lambda y} -1 +\lambda y)\lambda^{-\alpha\beta-1} \,d \lambda.
		\end{align} For all $\varphi \in \mc{S}_+$ and $n\ge 1$	
		\begin{align}\label{eq:mart_prob_Zn}
			&M^{Z,n}_{t}(\varphi)  = e^{-\langle \varphi, Z^{(n)}_{t} \rangle} - e^{- \langle \varphi,  Z^{(n)}_0 \rangle } 
			\\
			&- \int_0^{t}  e^{-\langle \varphi, Z^{(n)}_{s-}\rangle} \left(- \langle \frac{1}{2}\Delta \varphi, Z^{(n)}_s  \rangle  \nonumber + \langle \varphi(\cdot)^{\alpha\beta} -  \eta g \left(1/n, \varphi(\cdot)\right), (Z^{(n)}_{s-}(\cdot))^{\alpha}  \rangle \right) \,d s
		\end{align}
		is an $\mc{F}^{Z^{(n)}}$-local martingale with stopping times 
		\begin{align}\label{eq:gamma^Z,n}
			\gamma^{Z,n}(k) : = \gamma^n(k) = \inf\left\{ s\ge 0 \middle| \int_0^s \lVert Z^{(n)}_{r}\rVert^{\alpha}_{\alpha} \,d r >k \right\}, \quad k \in \N.
		\end{align}
	\end{lem}
	
	\begin{proof}
		See \cite[Lemma 3.7]{mytnik02}.
	\end{proof}
	
	We conclude the section with a result describing the behaviour of the compensator of $\mc{N}^{Z^{(n)}}= \mc{N}^n$ which is the counting measure tracking the jumps of $Z^{(n)}$. 
	
	\begin{lem}\label{lem:lem3.2_m98}
		The compensator of $\mc{N}^n$ is, for $B_1 \in \mc{B}([0,\infty)), B_2 \in \mc{B}(\R)$
		\begin{align}
			\hat{\mc{N}}^n(t, B_1\x B_2) = \eta \int_0^{t\wedge T^*_n}\,d r \int_{B_1}\,d \lambda \int_{B_2} \,d x  \frac{\left(Z^{(n)}_{r-}(x)\right)^{\alpha}}{\lVert Z^{(n)}_{r-} \rVert_{\alpha}^{\alpha}} \ind(\lambda >1/n) \lambda^{-\alpha\beta -1}
		\end{align}
		where \begin{align}\label{eq:def_T^Z,n,*}
			T^*_n = \inf\{t \ge 0 \mid \gamma^n(t)= \infty\} = \inf\{t\ge 0 \mid  A^{n}_t - b_n t = 0\}.
		\end{align}
	\end{lem}
	
	\begin{proof}
		See \cite[Lemma 3.5]{mytnik02}.
	\end{proof}	
	
	\section{Moment Estimate and Martingale Problem} \label{sec:moment_mart}
	
	In this section we will establish the key moment estimate for solutions $Y$ of \eqref{eq:main_spde_weak_form} and also show that $M^Y(\psi)$ defined in \eqref{eq:main_mart_problem} is a martingale for all $\psi \in \mc{S}_+$. The following is an alternative proof of the estimate presented in \cite[Lemma 2.4]{yz17}. Recall from the statement of Theorem \ref{thm:main_spde} that $Y_0 \in M_F$, the collection of all finite non-negative measures on $\R$. 

	\begin{prop}\label{prop:moment_est_Y_2}
		 Let $1<\alpha<2$ and $\frac{1}{\alpha}<\beta <1$. If $1\le q<\alpha$, then for a.e. $t \in [0,T]$ we have
		\begin{align}\label{eq:moment_est_Y_2.0}
			\sup_{x\in \R}\E(Y_t(x)^q) \le C \, t^{-\frac{q}{2}} + C
		\end{align}
		where $C = C(T,Y_0, \alpha, \beta)>0$ is a constant.
	\end{prop}
	
	\begin{rem}\label{rem:moment_1}
		We note in passing that when $Y_0$ is a bounded function on $\R$, the above estimate can be improved further. In this situation we will have, 
		\begin{align*}
			\sup_{x\in \R}\E(Y_t(x)^q) \le C'_1(T,Y_0) e^{C'_2(T)t}, \quad t \in [0,T],
		\end{align*}	
		where $C'_1$ and $C'_2$ are positive constants.
	\end{rem} 
	
	\begin{proof}[Proof of Proposition \ref{prop:moment_est_Y_2}]
		From \eqref{eq:main_spde_mild_PRM_2} we have, for $t \in [0,T]$,
		\begin{align}\label{eq:moment_Y_1}
			Y_t(x) = P_tY_0(x) + \int_{0}^t \int_{0}^{\infty}\int_{\R}\int_{0}^{Y_r(y)^{\alpha\beta}}z p_{t-r}(x-y)\tilde{N}_0(\, d v, \, dy, \, d z, \, d r).
		\end{align}
		
		Let us define
		\begin{align*}
			\tau_N = \inf\left\{ t\ge 0 \mid \int_0^t \lVert Y_r \rVert_{\alpha\beta}^{\alpha\beta} \,d r \ge N\right\},
		\end{align*}
		when $N \in \N$ and from \cite[Lemma 8.21]{pz07} recall that the quadratic variation of 
		\begin{align*}
			\int_{0}^s \int_{0}^{\infty}\int_{\R}\int_{0}^{Y_r(y)^{\alpha\beta}} z p_{t-r}(x-y) \tilde{N}_0(\, d v, \, dy, \, d z, \, d r)
		\end{align*} equals 
		\begin{align*}
			\int_{0}^s \int_{0}^{\infty}\int_{\R}\int_{0}^{Y_r(y)^{\alpha\beta}} z^2 p_{t-r}(x-y)^2 N_0(\, d v, \, dy, \, d z, \, d r),
		\end{align*}  
		for $s \in [0,T]$.  By the Burkholder-Davis-Gundy inequality (cf. \cite[Theorem IV.48]{protter04}) and the fact that $q<2$ we have
		\begin{align}\label{eq:moment_Y_1.1}
			I := &  \E \left[ \left\lvert \int_{0}^t \int_{0}^{\infty}\int_{\R}\int_{0}^{Y_r(y)^{\alpha\beta}} z p_{t-r}(x-y) \tilde{N}_0(\, d v, \, dy, \, d z, \, d r) \right\rvert^q \ind(t \le \tau_N)\right]
			\nonumber \\
			= & \E \left[ \left\lvert \int_{0}^t \int_{0}^{\infty}\int_{\R}\int_{0}^{Y_r(y)^{\alpha\beta}} z p_{t-r}(x-y)\ind(r \le \tau_N) \tilde{N}_0(\, d v, \, dy, \, d z, \, d r) \right\rvert^q \right]
			\nonumber \\
			\le & c \E \left[ \left\lvert \int_{0}^t \int_{0}^{\infty}\int_{\R}\int_{0}^{Y_r(y)^{\alpha\beta}} z^2 p_{t-r}(x-y)^2 \ind(r \le \tau_N) N_0(\, d v, \, dy, \, d z, \, d r) \right\rvert^{q/2} \right]
			\nonumber \\
			\le &  c \E \left[ \left\lvert \int_{0}^t \int_{0}^{1}\int_{\R}\int_{0}^{Y_r(y)^{\alpha\beta}} z^2 p_{t-r}(x-y)^2 \ind(r \le \tau_N) N_0(\, d v, \, dy, \, d z, \, d r) \right\rvert^{q/2} \right]
			\nonumber \\
			& + c \E \left[ \left\lvert \int_{0}^t \int_{1}^{\infty}\int_{\R}\int_{0}^{Y_r(y)^{\alpha\beta}} z^2 p_{t-r}(x-y)^2 \ind(r \le \tau_N) N_0(\, d v, \, dy, \, d z, \, d r) \right\rvert^{q/2} \right]. 
		\end{align}
		 Let $p \in (\alpha, 2)$ be fixed. Applying Jensen's inequality to the above (noting that $p/q >1$) we have
		\begin{align} \label{eq:moment_Y_1.2}
			I \le & c \E \left[  \left\lvert \int_{0}^t \int_{0}^{1}\int_{\R}\int_{0}^{Y_r(y)^{\alpha\beta}} z^2 p_{t-r}(x-y)^2 \ind(r \le \tau_N) N_0(\, d v, \, dy, \, d z, \, d r) \right\rvert^{p/2} \right]^{q/p}
			\nonumber \\
			&  + c \E \left[ \left\lvert \int_{0}^t \int_{1}^{\infty}\int_{\R}\int_{0}^{Y_r(y)^{\alpha\beta}} z^2 p_{t-r}(x-y)^2 \ind(r \le \tau_N) N_0(\, d v, \, dy, \, d z, \, d r) \right\rvert^{q/2} \right]
			\nonumber \\
			\le & c \E \left[  \int_{0}^t \int_{0}^{1}\int_{\R}\int_{0}^{Y_r(y)^{\alpha\beta}} z^p p_{t-r}(x-y)^p \ind(r \le \tau_N) N_0(\, d v, \, dy, \, d z, \, d r) \right]^{q/p}
			\nonumber \\
			& + c \E \left[ \int_{0}^t \int_{1}^{\infty}\int_{\R}\int_{0}^{Y_r(y)^{\alpha\beta}} z^q p_{t-r}(x-y)^q \ind(r \le \tau_N) N_0(\, d v, \, dy, \, d z, \, d r)  \right],			
		\end{align}
		where the second inequality above is due a fact about random sums (see the proof of \cite[Lemma 8.22]{pz07}). Now we use the definition of the PRM $N_0$, integrate out $z$ and use the inequality $u^{q/p}\le u +1$ for $u \ge 0$. 
		\begin{align}  \label{eq:moment_Y_1.3}
			I \le &  c+ c \E \left[  \int_{0}^t \int_{0}^{1}\int_{\R}\int_{0}^{Y_r(y)^{\alpha\beta}} z^p p_{t-r}(x-y)^p \ind(r \le \tau_N) N_0(\, d v, \, dy, \, d z, \, d r) \right] 
			\nonumber \\
			& + c \E \left[ \int_{0}^t \int_{1}^{\infty}\int_{\R}\int_{0}^{Y_r(y)^{\alpha\beta}} z^q p_{t-r}(x-y)^q \ind(r \le \tau_N) N_0(\, d v, \, dy, \, d z, \, d r)  \right]
			\nonumber \\			
			\le & c+  c  \E \left[  \int_0^{ t}\int_{\R} Y_r(y)^{\alpha\beta} p_{t-r}(x-y)^p \ind(r \le \tau_N) \, d y \, d r  \right]
			\nonumber \\
			& + c  \E \left[ \int_0^{t} \int_{\R}Y_r(y)^{\alpha\beta} p_{t-r}(x-y)^q  \ind(r \le \tau_N) 	\, d y  \, d r  \right]
			\nonumber \\
			\le & c + c \E \int_0^{ t} \int_{\R} (p_{t-r}(x-y)^p + p_{t-r}(x-y)^q) Y_r(y)^{\alpha\beta} \ind(r \le \tau_N) \, d y  \, d r,
		\end{align}
		as $\int_0^1 z^p m_0(dz)<\infty$ and $\int_0^{\infty} z^q m_0(dz) <\infty$.
		From \eqref{eq:moment_Y_1} and \eqref{eq:moment_Y_1.3} we have
		\begin{align}\label{eq:moment_Y_3}
			& \E \left[ Y_t(x)^q \ind(t<\tau_N)\right]
			\nonumber \\
			\le &c( P_tY_0(x))^q + c+  c\E \int_0^{t} \int_{\R} (p_{t-r}(x-y)^p +p_{t-r}(x-y)^q) Y_r(y)^{\alpha\beta} \ind(r<\tau_N)\, dr \, dy
			\nonumber \\
			= & c(P_tY_0(x))^q + c+  c \int_0^{t} \int_{\R} (p_{t-r}(x-y)^p +p_{t-r}(x-y)^q) \E \left[ Y_r(y)^{\alpha\beta} \ind(r<\tau_N)\right]\, dr \, dy.			
		\end{align}
		by applying Fubini's theorem in the last line. Use the definition of $p_t(x)$ to get
		\begin{align}\label{eq:moment_Y_4}
			& \E \left[ Y_t(x)^q  \ind(t<\tau_N)\right]
			\nonumber \\
			\le & c( P_tY_0(x))^q + c+  c \int_0^t \, dr ((t-r)^{-\frac{p-1}{2}}+(t-r)^{-\frac{q-1}{2}}) \int_{\R} p_{t-r}(x-y) \E\left[ Y_r(y)^{\alpha\beta} \ind(r<\tau_N)\right] \, dy 
			\nonumber \\
			\le & c( P_tY_0(x))^q + c+ c \int_0^t \, dr (t-r)^{-\frac{p-1}{2}} \int_{\R} p_{t-r}(x-y) \E\left[ Y_r(y)^{\alpha\beta} \ind(r<\tau_N)\right] \, dy  
		\end{align}	
		where in the last line we have used the fact that $(t-r)^{-\frac{q-1}{2}} \le C_T (t-r)^{-\frac{p-1}{2}}$. 
		
		When $q=\alpha\beta$ this becomes
		\begin{align}\label{eq:moment_Y_5}
			& \E \left[ Y_{t}(x)^{\alpha\beta} \ind(t<\tau_N)\right]
			\nonumber \\
			\le & c (P_tY_0(x))^{\alpha\beta} + c+ c \int_0^t \, dr (t-r)^{-\frac{p-1}{2}} \int_{\R} p_{t-r}(x-y) \E\left[ Y_r(y)^{\alpha\beta} \ind(r<\tau_N)\right] \, dy.		 			
		\end{align}		
		Let $s \in [0,T]$ be such that $s\ge t$. Apply $P_{s-t}$ to both sides and use Fubini's theorem,
		\begin{align}\label{eq:moment_Y_6}
			& \E \left[ P_{s-t}\left(Y_{t}^{\alpha\beta}\right)(x) \ind(t<\tau_N)\right]
			\nonumber \\
			\le & c(P_sY_0(x))^{\alpha\beta} + c+ c \int_0^t \, dr\, (t-r)^{-\frac{p-1}{2}} \int_{\R}  p_{s-t}(x-y)\int_{\R} p_{t-r}(y-z) \E\left[ Y_r(z)^{\alpha\beta} \ind(r<\tau_N)\right] \, dz \, dy 
			\nonumber \\
			\le &c(P_sY_0(x))^{\alpha\beta} + c+ c \int_0^t \, dr (t-r)^{-\frac{p-1}{2}} \int_{\R}  p_{s-r}(x-z)\E\left[ Y_r(z)^{\alpha\beta} \ind(r<\tau_N)\right] \, dz 
			\nonumber \\
			= & c(P_sY_0(x))^{\alpha\beta} + c+ c \int_0^t \, dr\, (t-r)^{-\frac{p-1}{2}} \E\left[ P_{s-r}\left(Y_r^{\alpha\beta}\right)(x) \ind(r<\tau_N)\right]
			\nonumber \\
			\le & c(Y_0) s^{-\frac{\alpha\beta}{2 }} + c+ c \int_0^t \, dr\, (t-r)^{-\frac{p-1}{2}} \E\left[ P_{s-r}\left(Y_r^{\alpha\beta}\right)(x) \ind(r<\tau_N)\right]
		\end{align}
		where we have used the assumption on $Y_0$ to obtain the bound on $(P_sY_0(x))^{\alpha\beta}$. The constants appearing hereafter all depend on $Y_0$. Since the above holds for every $t \in [0,s]$, by Lemma \ref{lem:gronwall} there exists a function $C_1$ on $(0,T]$ and a constant $C_2(s)>0$ such that for a.e. $t \le s$,
		\begin{align}\label{eq:moment_Y_7}
			\E \left[ P_{s-t}\left(Y_{t}^{\alpha\beta}\right)(x) \ind(t<\tau_N)\right] \le C_1(t) + \int_0^t C_1(r) e^{C_2(s)r} \, dr.
		\end{align} 
		
		Observe from the proof of Lemma \ref{lem:gronwall} that 
		\begin{align*}
			C_1(t) = o(t^{-\frac{\alpha\beta}{2}}) \text{ as } t \downarrow 0 
		\end{align*} and that the constant $C_2(s)$ is non-decreasing in $s$. So we have $C_2(s) \le C_2(T)$ and \eqref{eq:moment_Y_7} gives
		\begin{align} \label{eq:moment_Y_8}
			\E \left[ P_{s-t}\left(Y_{t}^{\alpha\beta}\right)(x) \ind(t<\tau_N)\right] \le C_3 t^{-\frac{\alpha\beta}{2}} + C_3\int_0^t r^{-\frac{\alpha\beta}{2}}  e^{C_2(T)r} \, dr.
		\end{align}
		for a.e. $t\le s \le T$. Here $C_3= C_3(T) >0$ is a constant. Now replace $s$ by $t$ in the above. We get,
		\begin{align}\label{eq:moment_Y_9}
			\E \left[ \left(Y_{t}(x)\right)^{\alpha\beta} \ind(t<\tau_N)\right]
			\le C_3 t^{-\frac{\alpha\beta}{2}} + C_3\int_0^t r^{-\frac{\alpha\beta}{2}}  e^{C_2(T)r} \, dr.				
		\end{align} 
		for a.e. $t \in [0,T]$. 
		
		We now plug this into \eqref{eq:moment_Y_4} to get,
		\begin{align}\label{eq:moment_Y_10}
			& \E \left[ Y_t(x)^q  \ind(t<\tau_N)\right]
			\nonumber \\
			\le & c \, t^{-\frac{q}{2}} + c+ c \int_0^t \, dr (t-r)^{-\frac{p-1}{2}} \int_{\R} p_{t-r}(x-y) \E\left[ Y_r(y)^{\alpha\beta} \ind(r<\tau_N)\right] \, dy
			\nonumber \\
			\le & c \, t^{-\frac{q}{2}} + c+ C_4 \int_0^t \, dr (t-r)^{-\frac{p-1}{2}} r^{-\frac{\alpha\beta}{2}} + C_4 \int_0^t \, dr (t-r)^{-\frac{p-1}{2}} r^{1-\frac{\alpha\beta}{2}}
			\nonumber \\
			= & c \, t^{-\frac{q}{2}} + c+ C_5 t^{1-\frac{p-1}{2} - \frac{\alpha\beta}{2}} + C_6 t^{2-\frac{p-1}{2} - \frac{\alpha\beta}{2}}
		\end{align}
		where $C_4 = C_4(T)>0$ is a constant and $C_5 = C_4 \, B(1-\frac{p-1}{2}, 1- \frac{\alpha\beta}{2})$, $C_6 =C_4 \, B(1-\frac{p-1}{2}, 2- \frac{\alpha\beta}{2})$ with $B$ denoting the Beta function.  At this point we consider the different regimes that the parameters $\alpha$, $\beta$ and $q$ can occupy. When $\frac{1}{\alpha}< \beta < (\frac{3}{\alpha}-1)\wedge 1$, one can find $p \in (\alpha ,2)$ such that $\frac{p-1}{2} + \frac{\alpha\beta}{2}\le 1$. Fix such a $p$ and observe that the exponent in the second term in RHS of \eqref{eq:moment_Y_10} is non-negative. This proves, 
			\begin{align}\label{eq:moment_Y_11}
				\E \left[ Y_t(x)^q  \ind(t<\tau_N)\right]
				\le C \, t^{-\frac{q}{2}} + C
			\end{align}
			where $C$ is constant depending on $T$, $Y_0$ and the parameters $\alpha, \beta$, when $0< \beta < (\frac{3}{\alpha}-1)\wedge 1$.
			
			Next, we consider the case when $\frac{3}{\alpha} -1<\beta<1$ (which requires $\frac{3}{2} <\alpha < 2$). Observe that $\alpha(1+\beta)-3< 2\alpha -3 <1.$ Since $q> 1$ by assumption, we have $\frac{q}{2} > \frac{\alpha-1}{2} + \frac{\alpha\beta}{2}-1$ and therefore there is a $p \in (\alpha, 2)$ such that $\frac{q}{2} > \frac{p-1}{2} + \frac{\alpha\beta}{2}-1$. Since \eqref{eq:moment_Y_10} holds for this $p$, so does \eqref{eq:moment_Y_11} for small enough $t$. 
			
			In \eqref{eq:moment_Y_11}, take $N \to \infty$ and we obtain the required result.
		
	\end{proof}
	
	We here observe that the previous moment estimate can be utilized to show that the stochastic integrals appearing in \eqref{eq:main_spde_mild}, \eqref{eq:main_spde_mild_PRM_1} and \eqref{eq:main_spde_mild_PRM_2} are martingales. For this we recall the notion of a \emph{class DL} process (see \cite[Definition IV.1.6]{ry99}).
	
	A real valued and adapted stochastic process $X$ is said to be of \emph{class DL} if for every $t>0$, the set $$\{X_\tau : \tau\le t \text{ is a stopping time}\}$$ is uniformly integrable. And we know from \cite[Proposition IV.1.7]{ry99} that a local martingale $X$ is a martingale if and only if it is of class DL. For practical purposes it is enough to show that there is an $\epsilon >0$ such that $$\sup_{\tau \le t}\E (|X_{\tau}|^{1+\epsilon}) < \infty$$ where the supremum is taken over all stopping times $\tau \le t$.  
	
	Using this we observe that $M^Y_t (\psi)$ defined in \eqref{eq:main_mart_problem} is a martingale. This will be crucial for simplifying our approximate duality argument in the proof of Proposition \ref{prop:lem3.1_m98}.
	
	\begin{prop}\label{prop:M^Y_mart}
		For each $\psi \in \mc{S}_+$, the local martingale $M^Y(\psi)$ is in fact a martingale with respect to $\mc{F}^Y$, the filtration generated by $Y$.
	\end{prop}	
	
	\begin{proof}
		Recall that 
		\begin{align*}
			M^Y_t(\psi)  = e^{-\langle Y_t, \psi\rangle} - e^{-\langle Y_0, \psi\rangle}  - \int_0^t I(Y_{s-}, \psi) \,d s
		\end{align*} 
		is an $\mc{F}^{Y}_t$-local martingale, where $$ I(Y_{s-}, \psi) = e^{- \langle Y_{s-}, \psi\rangle } \left( -\langle Y_{s-}, \frac{1}{2} \Delta \psi\rangle + \langle Y_{s-}^{\alpha \beta}, \psi^{\alpha} \rangle\right) $$
		
		To show that $M^Y_t(\psi)$ is a martingale we show that it is in class DL, i.e. for each $t >0$, 
		\begin{align}\label{eq:M^Y_mart_1}
			\sup_{\tau \le t }\E\left( |M^Y_{\tau}(\psi)|^{1+\epsilon} \right) < \infty
		\end{align} 
		for some $\epsilon>0$. The supremum ranges over all $\mc{F}^Y$-stopping times $\tau $ that are bounded by $t$.
		
		From the expression above it is enough to prove
		\begin{align}\label{eq:M^Y_mart_2}
			\sup_{\tau\le t}\E \left(  \left|\int_0^{\tau} I(Y_{s-}, \psi) \,d s\right|^{1+\epsilon} \right) < \infty.
		\end{align}
		
		Fix a stopping time $\tau \le t$. By Jensen's inequality
		\begin{align}\label{eq:M^Y_mart_3}
			\left|\int_0^{\tau} I(Y_{s-}, \psi) \,d s\right|^{1+\epsilon}  		
			= &  \left| \int_0^{\tau} e^{- \langle Y_{s-}, \psi\rangle } \left( -\langle Y_{s-}, \frac{1}{2} \Delta \psi\rangle + \langle Y_{s-}^{\alpha \beta}, \psi^{\alpha} \rangle\right)\,d s \right| ^{1+\epsilon} 
			\nonumber  \\
			= &\tau^{1+\epsilon}   	\left| \frac{1}{\tau} \int_0^{\tau} e^{- \langle Y_{s-}, \psi\rangle } \left( -\langle Y_{s-}, \frac{1}{2} \Delta \psi\rangle + \langle Y_{s-}^{\alpha \beta}, \psi^{\alpha} \rangle\right) \,d s \right| ^{1+\epsilon} 
			\nonumber \\
			\le & \tau^{\epsilon}  \int_0^{\tau}\left|  -\langle Y_{s-}, \frac{1}{2} \Delta \psi\rangle + \langle Y_{s-}^{\alpha \beta}, \psi^{\alpha} \rangle \right| ^{1+\epsilon}  \,d s 
			\nonumber \\
			\le & C_{\epsilon}t^{\epsilon}  \int_0^{t}\left(  |\langle Y_{s-}, \frac{1}{2} \Delta \psi\rangle|^{1+\epsilon}  + |\langle Y_{s-}^{\alpha \beta}, \psi^{\alpha} \rangle|^{1+\epsilon}  \right)  \,d s.
		\end{align}

		Let $0<\epsilon<\frac{1}{\beta}-1<\alpha -1 <1$. Again apply Jensen's inequality, Fubini's Theorem and Proposition \ref{prop:moment_est_Y_2}.
		\begin{align}\label{eq:M^Y_mart_4}
			\E \int_0^{t} |\langle Y_{s-}, \frac{1}{2} \Delta \psi\rangle|^{1+\epsilon}  \,d s 		
			= & \frac{1}{2^{1+\epsilon}} \E   \int_0^{\tau}  \left| \int_{\R}Y_{s-} (x)\Delta \psi(x) \,d x\right|^{1+\epsilon}\,d s
			\nonumber \\
			\le &  \frac{\lVert\Delta \psi\rVert_{1}^{1+\epsilon} }{2^{1+\epsilon}} \E  \int_0^{t}  \left| \frac{1}{\lVert\Delta \psi\rVert_{1} } \int_{\R}Y_{s-}(x) |\Delta \psi(x)| \,d x\right|^{1+\epsilon}\,d s
			\nonumber \\
			\le &  \frac{\lVert\Delta \psi\rVert_{1}^{\epsilon} }{2^{1+\epsilon}}  \int_0^{t}   \int_{\R} \E Y_{s-}(x)^{1+\epsilon} |\Delta \psi(x)| \,d x \,d s
			\nonumber \\ 
			\le &  C_T  \frac{\lVert\Delta \psi\rVert_{1}^{1+\epsilon} }{2^{1+\epsilon}}  \int_0^{t}   s^{-\frac{1+\epsilon}{2}} \,d s + C_T  \frac{\lVert\Delta \psi\rVert_{1}^{1+\epsilon} }{2^{1+\epsilon}} 
			\nonumber \\
			= &   C_T \frac{\lVert\Delta \psi\rVert_{1}^{1+\epsilon} }{2^{1+\epsilon}} (1+ t^{1-\frac{1+\epsilon}{2}}).  					
		\end{align}
		
		Similarly,   
		\begin{align}\label{eq:M^Y_mart_5}
			\E \int_0^{t} |\langle Y_{s-}^{\alpha\beta}, \frac{1}{2} \Delta \psi\rangle|^{1+\epsilon}  \,d s 	
			\le &  \lVert \psi^{\alpha}\rVert_{1}^{\epsilon} \int_0^{t}   \int_{\R} \E Y_{s-}(x)^{\alpha\beta(1+\epsilon)} \psi(x)^{\alpha} \,d x \,d s		
			\nonumber \\
			= &  C_T \frac{\lVert\psi^{\alpha}\rVert_{1}^{1+\epsilon} }{2^{1+\epsilon}}  \int_0^{t}   s^{-\frac{\alpha\beta(1+\epsilon)}{2}} \,d s  + C_T \frac{\lVert\psi^{\alpha}\rVert_{1}^{1+\epsilon} }{2^{1+\epsilon} } 
			\nonumber \\
			= &  C_{T, \psi, \alpha, \beta, \epsilon} (t^{1-\frac{\alpha\beta(1+\epsilon)}{2}} +1) 
		\end{align}
		
		Note that $1-\frac{1+\epsilon}{2} \ge 0$ and $1-\frac{\alpha\beta(1+\epsilon)}{2} \ge 0$ by our conditions on $\alpha$, $\beta$ and $\epsilon$. Plugging \eqref{eq:M^Y_mart_4} and \eqref{eq:M^Y_mart_5} in \eqref{eq:M^Y_mart_3} we get,
		\begin{align}\label{eq:M^Y_mart_6}
			\E \left(  \left|\int_0^{\tau} I(Y_{s-}, \psi) \,d s\right|^{1+\epsilon} \right)		
			\le C_{T,\psi, \alpha, \beta, \epsilon} (1+ t^{1-\frac{\alpha\beta(1+\epsilon)}{2}} + t^{1-\frac{1+\epsilon}{2}})
		\end{align}
		Taking supremum over all $\tau \le t$ gives \eqref{eq:M^Y_mart_2}.
	\end{proof}
	
	
	We now show the above result holds for $\psi:\R_+ \x \R \to \R$ satisfying certain assumptions. 
	
	\begin{prop}\label{prop:time_dep_mart_prob}
		Let $T \in (0, \infty)$. If $Y$ is a solution to the martingale problem \eqref{eq:main_mart_problem} and $\psi :[0,T]\times \R \to \R$ is such that 
		\begin{itemize}
			\item[(i)] The map $[0, T] \ni s \mapsto \psi_s  \in \bf{L}^{\eta}(\R)\cap \bf{L}^{\rho}(\R)$ is continuous, for some fixed $\eta \in (\frac{1}{\beta}, \alpha)$ and $\rho \in (\alpha, \frac{\alpha}{\beta}\wedge 2 )$. (Note that, as $\frac{1}{\alpha}<\beta < 1$ and $\alpha <2$, such $\eta$ and $\rho$ exist.)				
			
			\item[(ii)] $\sup_{s \le T} \left\lVert \frac{\partial}{\partial s}\psi_s \right\rVert_{\bf{L}^{\frac{\alpha\beta}{\alpha\beta-1}} (\R)} < \infty$.		
			
			\item[(iii)] The map $[0,T] \to \bf{L}^{\infty}(\R)$, $s \mapsto \frac{\partial^2}{\partial x^2} \psi_s$ is continuous.	
		\end{itemize}
		Then,		
		\begin{align}\label{eq:main_mart_problem_time_dependent}
			&\tilde{M}^Y_{t}(\psi)  = e^{-\langle Y_{t}, \psi_{t} \rangle }  -  e^{-\langle Y_0, \psi_0 \rangle} - \int_0^ {t} \tilde{I}(Y_{s-}, \psi_s) \,d s  
		\end{align} 
		is an $\mc{F}^{Y}_t$ martingale, where 
		\begin{align}\label{eq:tilde_I}
			\tilde{I}(Y_{s-}, \psi_s) =  e^{- \langle Y_{s-}, \psi_s\rangle } \left[ -\langle Y_{s-}, \frac{1}{2} \partial^2_{xx} \psi_s + \partial_s \psi_s \rangle + \langle Y_{s-}^{\alpha \beta}, \psi_s^{\alpha} \rangle  \right].
		\end{align}
	\end{prop}	
	
	This result is probably already known, but we could not find a self-contained proof in the literature. Therefore we present its proof in the Appendix \ref{app:time_dep_mart}.		
	
	\section{Overview of the Proof of Theorem \ref{thm:required_approx_duality_one_sided}}\label{sec:proof_thm}
	
	In this section we describe our plan for proving Theorem \ref{thm:required_approx_duality_one_sided}. Our proof follows the argument in \cite{mytnik98} and will be split into various propositions which we state in the following. At the end of this section we establish the theorem assuming these results.     	
	
	The first proposition describes the behaviour of $Y$ when coupled with the solutions of the evolution equations used to construct $Z^{(n)}$. In what follows we denote by $\E_Y$ the expectation with respect to $Y$. In particular, under $\E_Y$ we treat all the random variables used to construct $Z^{(n)}$ in Section \ref{sec:prelim} as non-random owing to our assumption of independence.
	
	\begin{prop}\label{prop:lem3.1_m98}
		Let $Y$ be a solution to the martingale problem \eqref{eq:main_mart_problem}. Then for each  $t \in [0,T]$, $ n \ge 1$ and $\mu \in M_F$,
		\begin{align}\label{eq:(3.5)_m98}
			\E_Y \left[e^{-\langle Y_{T-t}, V^n_{t}(\mu)\rangle}	\right]
			=   \E_Y \left[e^{-\langle Y_{T}, V^n_{0}(\mu) \rangle}\right] + \E_Y\left[ \int^{t}_{0} \tilde{\mc{I}}(Y_{(T-r)-}, V^n_r(\mu))\, dr\right]
		\end{align}		
		where 
		\begin{align*}
			\tilde{\mc{I}}(Y_{(T-r)-}, V^n_r(\mu)) = e^{-\langle Y_{(T-r)-}, V^n_r(\mu) \rangle}  \left\{ -\langle Y^{\alpha\beta}_{(T-r)-}, \left(V^n_r(\mu)\right)^{\alpha} \rangle + \langle Y_{(T-r)-}, b_n \left(V^n_r(\mu)\right)^{\alpha}\rangle \right\},
		\end{align*}
		and $V^n$ is the solution of the PDE \eqref{eq:pde_Z}.
	\end{prop}
	
	In the next proposition we describe the relationship between $Y$ and the jumps of $Z^{(n)}$. Define 
	\begin{align}\label{eq:tau^n}
		\tau^n(t) : = \int_0^t \lVert Z^{(n)}_{r-}\rVert^{\alpha}_{\alpha}\,d r		
	\end{align} 
	and observe from \eqref{eq:gamma^Z,n} that $\tau^n$ is the inverse of $\gamma^n$: $\tau^n(\gamma^n(t)) = t$ and vice-versa.
	\begin{prop}\label{prop:(3.19)_m98}
		If $Y$ is a solution to the martingale problem (\ref{eq:main_mart_problem}), independent of $Z^{(n)}$'s, then for all $t \in [0, T]$,
		\begin{align}\label{eq:(3.19)_m98}
			& \E_Y \left[e^{-\langle Y_{T-t}, Z^{(n)}_t\rangle} \right]
			\nonumber \\
			= &   \E_Y\left[e^{-\langle Y_T, Z_0 \rangle} + \int_{0}^{t}\tilde{\mc{I}}(Y_{(T-r)-}, Z^{(n)}_{r-} ) \,d r 
			+	\int_0^{\tau^n(t)} \int_{\R} \int_{\R_+} \theta_n(s,x, \lambda) \mc{N}^n(d \lambda, \, d x, \, d s)	\right] 			
		\end{align}		
		where
		\begin{align*}
			\theta_n(s,x, \lambda) = e^{-\langle Y_{T-\gamma^n(s)} , Z^{(n)}_{\gamma^n(s) -}\rangle} \left( e^{- \lambda Y_{T-\gamma^n(s)}(x) } -1 \right).
		\end{align*}
	\end{prop}
	
	In the last proposition before we prove our main result we show that the previous result holds at the stopping time $\Upsilon^n_{k}(t) := \gamma^n(k)\wedge t$. Recall the definitions of $\eta$ and $g$ from Lemma \ref{lem:lem3.7_myt02}.
	
	\begin{prop}\label{prop:lem3.3_m98}
		If $Y$ is a solution to the martingale problem (\ref{eq:main_mart_problem}), independent of $Z^{(n)}$'s, then for each $m \in \N$ and $t \in [0,T]$,
		\begin{align}\label{eq:(3.18)_m98}
			\E [\exp(-\langle Y_{T-\Upsilon^n_m(t)}, Z_{\Upsilon^n_m(t)}\rangle)]
			= &  \E [\exp\left(-\langle Y_T, Z_0 \rangle\right)]
			\nonumber \\						
			& - \eta  \E \left[ \int^{\Upsilon^n_m(t)}_{0} e^{-\langle Y_{(T-r)-}, Z^{(n)}_{r-}\rangle} \langle g (1/n, Y_{(T-r)-}(\cdot)), (Z^{(n)}_{r-})^{\alpha} \rangle  \,d r\right].			
		\end{align}		
	\end{prop}  
	We note that Propositions \ref{prop:lem3.1_m98} and \ref{prop:(3.19)_m98} are used to prove Proposition \ref{prop:lem3.3_m98}. Now we present the proof of Theorem \ref{thm:required_approx_duality_one_sided} assuming that the above propositions hold. We will prove them in the next section.
	
	\begin{proof}[Proof of Theorem \ref{thm:required_approx_duality_one_sided}]
		Let $Y$ and $Z^{(n)}$ be as in the statement of Theorem \ref{thm:required_approx_duality_one_sided}. Let  
		\begin{align*}
			k_n = \ln n.
		\end{align*} We will show that for a.e. $t \in [0,T]$, 
		\begin{align}\label{eq:thm_prf_0}
			\lim_{n\to \infty} \lvert \E [e^{-\langle Y_0, Z^{(n)}_{\Upsilon^n_{k_n}(t)}\rangle}] - \E[e^{-\langle Y_t, Z_0\rangle}]\rvert =0.
		\end{align}
		This will prove the theorem with the approximate dual processes being $\tilde{Z}^{(n)}_t := Z^{(n)}_{\Upsilon^n_{k_n}(t)}$.
		
		Towards this, we are first going to show that
		\begin{align}\label{eq:thm_prf_0.5}
			\left|\E \exp\left(-\langle Y_{T-\Upsilon^n_{k_n}(t)}, Z^{(n)}_{\Upsilon^n_{k_n}(t)}\rangle\right) - \E e^{-\langle Y_T, Z_0 \rangle}\right|		
			\le  C_{\alpha. \beta, T} ((T-t)^{-\frac{\bar{p}}{2}} +1) n^{- \frac{\alpha - \alpha \beta}{2}} k_n,
		\end{align}
		when $0\le t < T$. Note that, as $k_n = \ln n$, the RHS converges to $0$ as $n \to \infty$. 		
		
		Note that for all $1< p  < 2$ and $\lambda \ge 0$, $$ e^{-\lambda} -1 + \lambda \le \frac{\lambda^p}{p}.$$ Also by our assumptions on $\alpha$ and $\beta$ we have $1<\frac{\alpha(\beta+1)}{2}<\alpha<2$. So,
		\begin{align}\label{eq:lem3.4_m98_2}
			g\left(\frac{1}{n}, Y_{T-s}(x)\right) & = \int_{0+}^{1/n} \left( e^{-\lambda Y_{T-s}(x)} -1 + \lambda Y_{T-s}(x)\right) \lambda^{-\alpha \beta -1}\,d \lambda
			\nonumber \\	
			& \le \frac{2}{\alpha(\beta +1)} \int_{0+}^{1/n} \left(\lambda Y_{T-s}(x)\right)^{\frac{\alpha(\beta +1)}{2}} \lambda^{-\alpha \beta -1}\,d \lambda
			\nonumber \\	
			& = \frac{2}{\alpha(\beta +1)} Y_{T-s}(x)^{\frac{\alpha(\beta +1)}{2}} \int_{0+}^{1/n} \lambda^{\frac{\alpha(\beta +1)}{2} -\alpha \beta -1}\,d \lambda 
			\nonumber \\
			&= \frac{2}{\alpha(\beta +1)} Y_{T-s}(x)^{\frac{\alpha(\beta +1)}{2}} \frac{2}{\alpha - \alpha \beta} n^{- \frac{\alpha - \alpha \beta}{2} }.
		\end{align}
		
		Eq. \eqref{eq:(3.18)_m98} and the above calculation gives us 
		\begin{align}\label{eq:lem3.4_m98_3}
			& \left|\E \exp\left(-\langle Y_{T-\Upsilon^n_{k_n}(t)}, Z^{(n)}_{\Upsilon^n_{k_n}(t)}\rangle\right) - \E e^{-\langle Y_T, Z_0 \rangle}\right| 
			\nonumber \\ 
			=& \left\lvert  \eta  \E \left[ \int^{\Upsilon^n_{k_n}(t)}_{0+} e^{-\langle Y_{(T-s)-}, Z^{(n)}_s\rangle} \langle g(1/n, Y_{(T-s)-}(\cdot)), \left(Z^{(n)}_{s-}\right)^{\alpha} \rangle  \,d s\right] \right\rvert
			\nonumber  \\	
			= &  \eta  \E \left[ \int_0^{\Upsilon^n_{k_n}(t)} \int_{\R} Z^{(n)}_{s-}(x)^{\alpha} g(1/n, Y_{T-s}(x)) \,d x \,d s \right] 
			\nonumber \\	
			\le & \eta \frac{2}{\alpha(\beta +1)} \frac{2}{\alpha - \alpha \beta} n^{- \frac{\alpha - \alpha \beta}{2}}   \E \left[ \int_0^{\Upsilon^n_{k_n}(t)} \int_{\R} Z^{(n)}_{s-}(x)^{\alpha} Y_{T-s}(x)^{\frac{\alpha(\beta +1)}{2}} \,d x \,d s \right]
		\end{align}
		using the fact that $Z^{(n)}_{s-}(\cdot)\ge 0$ and $Y_{(T-s)-}(\cdot)\ge 0$ for the second equality. Now use the estimate from Proposition \ref{prop:moment_est_Y_2} with $\bar{p} = \frac{\alpha(\beta+1)}{2}$. We have by Fubini's theorem
		\begin{align}	
			\E \left[ \int_0^{\Upsilon^n_{k_n}(t)} \int_{\R} Z^{(n)}_{s-}(x)^{\alpha} Y_{T-s}(x)^{\frac{\alpha(\beta +1)}{2}} \,d x \,d s \right]				
			= &  \E_Z \E_Y \left[ \int_0^{\Upsilon^n_{k_n}(t)} \int_{\R} Z^{(n)}_{s-}(x)^{\alpha} Y_{T-s}(x)^{\bar{p}} \,d x \,d s\right] 
			\nonumber \\
			= &  \E_Z  \left[ \int_0^{\Upsilon^n_{k_n}(t)} \int_{\R} Z^{(n)}_{s-}(x)^{\alpha} \E_Y \left(Y_{T-s}(x)^{\bar{p}}\right) \,d x \,d s\right] 
			\nonumber \\
			\le &  C \E_Z \left[ \int_0^{\Upsilon^n_{k_n}(t)} \int_{\R} Z^{(n)}_{s-}(x)^{\alpha} (T-s)^{-\frac{\bar{p}}{2}} \,d x \,d s\right]  
			\nonumber \\
			& + C \E_Z \left[ \int_0^{\Upsilon^n_{k_n}(t)} \int_{\R} Z^{(n)}_{s-}(x)^{\alpha}  \,d x \,d s\right] 			
			\nonumber \\
			\le &    C \E_Z \left[ (T-\Upsilon^n_{k_n}(t))^{-\frac{\bar{p}}{2}} \int_0^{\Upsilon^n_{k_n}(t)} \lVert Z^{(n)}_{s-}\rVert^{\alpha}_{\alpha}  \,d s\right] 
			\nonumber \\
			& + C \E_Z \left[ \int_0^{\Upsilon^n_{k_n}(t)}  \lVert Z^{(n)}_{s-}\rVert^{\alpha}_{\alpha} \,d s\right]
			\nonumber \\
			\le &  C ((T-t)^{-\frac{\bar{p}}{2}} +1)  \E_Z \left[  \int_0^{\Upsilon^n_{k_n}(t)} \lVert Z^{(n)}_{s-}\rVert^{\alpha}_{\alpha}  \,d s\right] 
			\nonumber \\			
			\le & C ((T-t)^{-\frac{\bar{p}}{2}} +1) k_n .
		\end{align} 
		The third inequality is due to the fact that $\Upsilon^n_k (t) = \gamma^n(k)\wedge t \le t$ and the last inequality follows from the definition of $\gamma^n$ (see \eqref{eq:gamma^Z,n}). Plugging this in \eqref{eq:lem3.4_m98_3} gives \eqref{eq:thm_prf_0.5}.
		
		Next we turn our attention to \eqref{eq:thm_prf_0}. We can write,
		\begin{align} \label{eq:thm_prf_2}
			& \lvert \E\exp(-\langle Y_0, Z^{(n)}_{\Upsilon^n_{k_n}(t)} \rangle) - \E\exp(-\langle Y_t, Z_0\rangle) \rvert
			\nonumber \\
			\le & \lvert \E\exp(-\langle Y_0, Z^{(n)}_{\Upsilon^n_{k_n}(t)} \rangle) - \E\exp(-\langle Y_{t - \Upsilon^n_{k_n}(t-\frac{1}{k_n})}, Z^{(n)}_{\Upsilon^n_{k_n}(t-\frac{1}{k_n})}\rangle) \rvert 			 
			\nonumber \\
			& + \lvert \E\exp(-\langle Y_{t - \Upsilon^n_{k_n}(t-\frac{1}{k_n})}, Z^{(n)}_{\Upsilon^n_{k_n}(t-\frac{1}{k_n})}\rangle)  - \E\exp(-\langle Y_t, Z_0\rangle) \rvert .			
		\end{align}
		By \eqref{eq:thm_prf_0.5} (with $T$ and $t$ replaced by $t$ and $t-\frac{1}{k_n}$ respectively) we can bound the second term in the RHS of the above as follows,
		\begin{align*}
			\left|\E \exp\left(-\langle Y_{t-\Upsilon^n_{k_n}(t-\frac{1}{k_n})}, Z^{(n)}_{\Upsilon^n_{k_n}(t-\frac{1}{k_n})}\rangle\right) - \E e^{-\langle Y_t Z_0 \rangle}\right|		
			\le  C_{\alpha. \beta, t} (k_n^{\frac{\bar{p}}{2}} +1) n^{- \frac{\alpha - \alpha \beta}{2}} k_n.
		\end{align*} 
		We note that the RHS of the above converges to $0$ as $n \to \infty$.
		
		Let us now consider the first term in the RHS of \eqref{eq:thm_prf_2}. By definition of $\Upsilon^n$ and $T^*_n$ (see \eqref{eq:def_T^Z,n,*}) ,
		\begin{align*}
			& \lvert \E\exp(-\langle Y_{t - \Upsilon^n_{k_n}(t-\frac{1}{k_n})}, Z^{(n)}_{\Upsilon^n_{k_n}(t-\frac{1}{k_n})} \rangle) - \E\exp(-\langle Y_0, Z^{(n)}_{\Upsilon^n_{k_n}(t)} \rangle) \rvert
			\\
			= & \lvert \E\left[\exp(-\langle Y_{t - \Upsilon^n_{k_n}(t-\frac{1}{k_n})}, Z^{(n)}_{\Upsilon^n_{k_n}(t-\frac{1}{k_n})} \rangle) - \exp(-\langle Y_0, Z^{(n)}_{\Upsilon^n_{k_n}(t)} \rangle); \Upsilon^n_{k_n}(t) < t-\frac{1}{k_n}\right] \rvert
			\\
			& + \lvert \E\left[\exp(-\langle Y_{t - \Upsilon^n_{k_n}(t-\frac{1}{k_n})}, Z^{(n)}_{\Upsilon^n_{k_n}(t-\frac{1}{k_n})} \rangle) - \exp(-\langle Y_0, Z^{(n)}_{\Upsilon^n_{k_n}(t)} \rangle); \Upsilon^n_{k_n}(t) \ge t-\frac{1}{k_n}\right] \rvert
			\\
			\le & \pr(\Upsilon^n_{k_n}(t) < t -\frac{1}{k_n}) + \lvert \E\left[\exp(-\langle Y_{\frac{1}{k_n}}, Z^{(n)}_{t-\frac{1}{k_n}} \rangle) - \exp(-\langle Y_0, Z^{(n)}_{\Upsilon^n_{k_n}(t)} \rangle); \Upsilon^n_{k_n}(t) \ge t-\frac{1}{k_n}\right] \rvert .
		\end{align*}
		The second term above converges to $0$ since $Y$ is right-continuous and $\Upsilon^n_{k_n}(t) = \gamma^n(k_n) \wedge t \to t$ as $n \to \infty$. Also as $\pr(T^*_n < \infty)=1$ (see \cite[eq. (3.14)]{mytnik02}), we have 
		\begin{align*}
			\pr(\Upsilon^n_{k_n}(t) < t -\frac{1}{k_n})  = \pr(\gamma^n(k_n) < t-\frac{1}{k_n})  \le \pr (T^*_n > k_n) \to 0\text{ as } n \to \infty.
		\end{align*} This proves \eqref{eq:thm_prf_0}.				 
	\end{proof}
	
	\section{Proofs of key Propositions} \label{sec:proof_props}	
	
	We will prove the three propositions required for the proof of Theorem \ref{thm:required_approx_duality_one_sided} in this section. For Proposition \ref{prop:lem3.1_m98} we start by verifying \eqref{eq:(3.5)_m98} for measures having densities and then prove it for the case of general measures.
	
	\begin{proof}[Proof of Proposition \ref{prop:lem3.1_m98}]
		Let $\varphi_l \in \mc{S}(\R)_+$, $l \in \N$, be such that $\mu_l(d x) := \varphi_l(x)\,d x \implies \mu(dx)$ as $l \to \infty$. Since $n$ is fixed in this proof, let $v_l(\cdot) = V^n_{\cdot}(\mu_l)$ solve 
		\begin{align}\label{eq:lem3.1_m98_1}
			\partial_t v_l(t) & = \frac{1}{2}\partial^2_{xx} v_l(t) -  b_n v_l(t)^{\alpha} \nonumber\\
			v_l(0) & = \varphi_l.
		\end{align}
		Fix $l, k \in \N$. Let $\psi(s,x) := v_l(T-s, x) = V^n_{T-s}(\mu_l)(x)$. Lemma \ref{lem:pde_Z_properties} says that $\psi$ satisfies the conditions of Proposition \ref{prop:time_dep_mart_prob}.  	
		
		From \eqref{eq:tilde_I} recall that
		\begin{align*}
			\tilde{I}(Y_{s-}, \psi_s) =  e^{-\langle Y_{s-}, \psi_s \rangle}  \left[ -\langle Y_{s-}, \frac{1}{2}\Delta\psi_s \rangle + \langle Y_{s-}^{\alpha \beta},\psi_s^{\alpha} \rangle - \langle Y_{s-}, \frac{\partial}{\partial s}\psi_s\rangle  \right].
		\end{align*}	
		Then by Proposition \ref{prop:time_dep_mart_prob} for each $k \in \N$ and $t \in [0,T)$, 	
		\begin{align*}
			\E_Y \left[\tilde{M}^Y_{T-t}(\psi)\right] = \E_Y \left[\tilde{M}^Y_{T}(\psi)\right]
		\end{align*}
		which implies,		
		\begin{align}\label{eq:lem3.1_m98_2}
			\E_Y \exp\left(-\langle Y_{T-t},\psi_{T-t} \rangle \right)   
			= \E_Y \left[e^{-\langle Y_{T},\psi_{T} \rangle} - \int_{T-t}^{T} \tilde{I}(Y_{s-}, \psi_s)\, ds\right].
		\end{align}
		
		From the above definition of $\tilde{I}(Y,\psi)$ and \eqref{eq:lem3.1_m98_1} we get, 
		\begin{align}\label{eq:lem3.1_m98_3}
			& \int_{T-t}^{T} \tilde{I}(Y_{s-}, \psi_s) \, d s 
			\nonumber \\
			& =  \int_{T-t}^{T} e^{-\langle Y_{s-}, \psi_s \rangle}  \left\{ -\langle Y_{s-}, \frac{1}{2}\Delta\psi_s \rangle + \langle Y_{s-}^{\alpha \beta},\psi_s^{\alpha} \rangle - \langle Y_{s-}, \frac{\partial}{\partial s}\psi_s\rangle  \right\} \,d s
			\nonumber \\
			& = \int^{T}_{T-t} e^{-\langle Y_{s-}, v_l(T-s) \rangle}  \left\{ -\langle Y_{s-}, \frac{1}{2}\Delta v_l(T-s) \rangle + \langle Y_{s-}^{\alpha \beta},v_l(T-s)^{\alpha} \rangle - \langle Y_{s-}, \frac{\partial}{\partial s}v_l(T-s)\rangle  \right\} \,d s 
			\nonumber \\
			& =  - \int^{0}_{t} e^{-\langle Y_{(T-r)-}, v_l(r) \rangle}  \left\{ -\langle Y_{(T-r)-}, \frac{1}{2}\Delta v_l(r) \rangle + \langle Y_{(T-r)-}^{\alpha \beta},v_l(r)^{\alpha} \rangle + \langle Y_{(T-r)-}, \frac{\partial}{\partial r}v_l(r)\rangle  \right\} \,d r, 
			\nonumber \\
			& = -\int_{t}^{0} e^{-\langle Y_{(T-r)-},v_l(r) \rangle}  \left\{ \langle Y_{(T-r)-}^{\alpha \beta},v_l(r)^{\alpha} \rangle + \langle Y_{(T-r)-}, \frac{\partial}{\partial r}v_l(r) - \frac{1}{2}\Delta v_l(r)\rangle  \right\} \,d r 
			\nonumber \\
			& = -\int_{t}^{0} e^{-\langle Y_{(T-r)-}, v_l(r) \rangle}  \left\{ \langle Y_{(T-r)-}^{\alpha \beta},v_l(r)^{\alpha} \rangle - \langle Y_{(T-r)-},  b_n  v_l(r)^{\alpha} \rangle  \right\} \,d r,  
		\end{align}
		using the substitution $r = T-s$ for the third equality.
		
		By \eqref{eq:lem3.1_m98_2} and \eqref{eq:lem3.1_m98_3}, 
		\begin{align}\label{eq:lem3.1_m98_4}
			& \E_Y \exp\left(-\langle Y_{T-t}, v_l(t) \rangle \right)  =  \E_Y \exp\left(-\langle Y_{T}, v_l(0) \rangle \right) 		
			+ \E_Y \int_{0}^{t} \tilde{\mc{I}}(Y_{(T-r)-},v_l(r)) \, dr   
		\end{align}
		where $$\tilde{\mc{I}}(Y_{(T-r)-}, v_l(r)) = e^{-\langle Y_{(T-r)-}, v_l(r) \rangle} \frac{1}{2} \left( \langle Y_{(T-r)-}, b_n \left(v_l(r)\right)^{\alpha}\rangle - \langle Y^{\alpha\beta}_{(T-r)-}, \left(v_l(r)\right)^{\alpha} \rangle \right) $$ 
		
		We now have to check whether this holds when $\mu: = w-\lim_{l \to \infty} \mu_l$. 
		
		Let $v(r) = V^n_r(\mu)$,
		\begin{align*}
			R_l:=  \E_Y \int^{t}_{0}\tilde{\mc{I}}(Y_{(T-r)-}, V^n_r(\mu_l)) \,d r
			= \E_Y\int_{0}^{t}\tilde{\mc{I}}(Y_{(T-r)-}, v_l(r))) \,d r,
		\end{align*} 
		and
		\begin{align*}
			R := \E_Y \int^{t}_{0}\tilde{\mc{I}}(Y_{(T-r)-}, V^n_r(\mu)) \,d r
			=  \E_Y\int^{t}_{0}\tilde{\mc{I}}(Y_{(T-r)-},v(r)) \,d r.
		\end{align*} 
		
		We only have to prove $R_l \to R$ as $l \to \infty$. In $R_l - R$ adding and subtracting the term $$e^{-\langle Y_{(T-r)-}, v(r)\rangle} \left(b_n \langle Y_{(T-r)-}, v_l(r)^{\alpha}\rangle - \langle Y_{(T-r)-}^{\alpha\beta}, v_l(r)^{\alpha}\rangle \right)$$ we have	
		\begin{align*}
			|R_l - R| \le  \frac{1}{2} (I_1^{l} + I_2^{l}),
		\end{align*}
		where	
		\begin{align*}
			I_1^{l} = & \E_Y\left| \int_0^t \left(e^{-\langle Y_{(T-r)-}, v_l(r)\rangle} - e^{-\langle Y_{(T-r)-}, v(r)\rangle}\right) \left(b_n \langle Y_{(T-r)-}, v_l(r)^{\alpha}\rangle - \langle Y_{(T-r)-}^{\alpha\beta}, v_l(r)^{\alpha}\rangle \right) \,d r \right|
			\\
			I_2^{l}= & \E_Y\left| \int_0^t  e^{-\langle Y_{(T-r)-}, v(r)\rangle} \left(b_n \langle Y_{(T-r)-}, v_l(r)^{\alpha}-v(r)^{\alpha}\rangle  
			- \langle Y_{(T-r)-}^{\alpha\beta}, v_l(r)^{\alpha}-v(r)^{\alpha}\rangle \right) \,d r \right|.
		\end{align*}
		
		To prove $|R_l - R| \to 0$ as $l \to \infty$ we need to show
		
		\textbf{(i):} $I_1^{l} \to 0$; and 
		
		\textbf{(ii):} $I_2^{l}\to 0$ as $ l \to \infty$.\\
		
		\textbf{Proof of (i): }	 	Let $1< q <\frac{1}{\beta}$ and $p>1$ be such that $\frac{1}{p}+\frac{1}{q}=1$. Note that 
		\begin{align}\label{eq:lem3.1_m98_5}
			I_1^{l} &\le \int_0^t \E_Y \left| \left(e^{-\langle Y_{(T-r)-}, v_l(r)\rangle} - e^{-\langle Y_{(T-r)-}, v(r)\rangle}\right) \left(b_n \langle Y_{(T-r)-}, v_l(r)^{\alpha}\rangle - \langle Y_{(T-r)-}^{\alpha\beta}, v_l(r)^{\alpha}\rangle \right) \right| \,d r
			\nonumber \\	
			&\le \int_0^t \E_Y\left(\left| e^{-\langle Y_{(T-r)-}, v_l(r)\rangle} - e^{-\langle Y_{(T-r)-}, v(r)\rangle} \right|^p\right)^{1/p} \E_Y\left(\left|b_n \langle Y_{(T-r)-}, v_l(r)^{\alpha}\rangle - \langle Y_{(T-r)-}^{\alpha\beta}, v_l(r)^{\alpha}\rangle \right|^q\right)^{1/q} \,d r
			\nonumber \\
			& =: \int_0^t I^{l}_{11}(r) I^{l}_{12}(r) \,d r
		\end{align}	
		using H\"{o}lder's inequality in the second line. Here $I^{l}_{11}(r)$ and $I^{l}_{12}(r)$ denote the first and second terms of the integrand in the above. 
		
		Now let us use a notation from Fleischmann \cite{fleisch86}:
		\begin{align*}
			\lVert v \rVert_{\bf{L}^{\alpha,T}} : = \sup_{0\le t \le T}  \lVert v(t) \rVert_{\bf{L}^{\alpha}(\R)}.	
		\end{align*}	
		By \cite[Proposition A2]{fleisch86}, we have $v_l \to v$ in $\bf{L}^{\alpha,T}$ as $l \to \infty$. Thus there is a subsequence of $v_l$, which we also denote as $v_l$ by a slight abuse of notation, such that $v_l(t,x) \to v(t,x)$ as $l \to \infty$ for a.e. $t \in [0,T]$ and $x \in \R$. As the term inside the expectation of $I^{l}_{11}(r)$ is bounded by $2$, the dominated convergence theorem gives us
		\begin{align*}
			\lim_{l \to \infty }I^{l}_{11}(r)= 0
		\end{align*} for each $r \in [0,t]$. Since $|I^{l}_{11}(r)| \le 2$ for all $l$ and $r$, again by the dominated convergence theorem, to prove (i) as above we only have to show that $I^{l}_{12}(r)  \le C < \infty$ for some constant $C = C_t$ independent of $l$. 
		
		\begin{align}\label{eq:lem3.1_m98_6}
			I^{l}_{12}(r) & = \E_Y\left(\left|b_n \langle Y_{(T-r)-}, v_l(r)^{\alpha}\rangle - \langle Y_{(T-r)-}^{\alpha\beta}, v_l(r)^{\alpha}\rangle \right|^q\right)^{1/q} 
			\nonumber \\
			& \le b_n \E_Y\left(\left| \langle Y_{(T-r)-}, v_l(r)^{\alpha}\rangle \right|^{q}\right)^{1/q}+ \E_Y \left(\langle Y_{(T-r)-}^{\alpha\beta}, v_l(r)^{\alpha}\rangle^q\right)^{1/q} 
			\nonumber \\
			& := b_n I_{121}^{l}(r) + I_{122}^{l}(r)
		\end{align}
		using Minkowski's inequality.
		
		For all $r <t$,
		\begin{align}\label{eq:lem3.1_m98_7}
			I_{121}^{l}(r) = & \E_Y\left(\left| \langle Y_{(T-r)-}, v_l(r)^{\alpha}\rangle \right|^{q}\right)^{1/q}
			\nonumber \\
			= & \lVert v_l(r) \rVert_{\bf{L}^{\alpha}(\R)}^{\alpha} \E_Y\left[ \left( \int_{\R} \frac{1}{ \lVert v_l(r) \rVert_{\bf{L}^{\alpha}(\R)}^{\alpha}} v_l(r,x)^{\alpha} Y_{T-r}(x) \, d x \right)^{q}\right]^{1/q}
			\nonumber \\
			\le & \lVert v_l(r) \rVert_{\bf{L}^{\alpha}(\R)}^{\alpha}  \left[ \E_Y  \left( \frac{1}{ \lVert v_l(r) \rVert_{\bf{L}^{\alpha}(\R)}^{\alpha}}  \int_{\R} v_l(r,x)^{\alpha} Y_{T-r}(x)^{q}\,d x \right) \right]^{1/q}
			\nonumber \\
			= & \lVert v_l(r) \rVert_{\bf{L}^{\alpha}(\R)}^{\alpha-\alpha/q} \left[ \int_{\R} v_l(r,x)^{\alpha} \E_Y (Y_{T-r}(x)^{q}) \,d x \right]^{1/q}
			\nonumber \\
			\le & C_T \lVert v_l(r) \rVert_{\bf{L}^{\alpha}(\R)}^{\alpha - \alpha/q} \left[ \int_{\R} v_l(r,x)^{\alpha} (T-r)^{-\frac{q}{2}} \,d x +  \int_{\R} v_l(r,x)^{\alpha}  \,d x \right]^{1/q}
			\nonumber \\
			\le & C_T \lVert v_l(r) \rVert_{\bf{L}^{\alpha}(\R)}^{\alpha - \alpha/q} ((T-t)^{-\frac{q}{2}}+1)^{1/q} \left[ \int_{\R} v_l(r,x)^{\alpha}  \,d x \right]^{1/q}
			\nonumber \\
			= & C_T \lVert v_l(r) \rVert_{\bf{L}^{\alpha}(\R)}^{\alpha}((T-t)^{-\frac{q}{2}}+1)^{1/q}.		
		\end{align}
		Here we have used Jensen's inequality and Proposition \ref{prop:moment_est_Y_2} (applicable by our assumption that $q < \alpha$) in the first and second inequalities respectively. \cite[Proposition A2]{fleisch86} implies that for large enough $l \in \N$, $\lVert v_l (r)\rVert_{\bf{L}^{\alpha}(\R)} \le \lVert v\rVert_{\bf{L}^{\alpha,T}} +1$ for all $r \in [0,t]$. Therefore \eqref{eq:lem3.1_m98_7} gives us 
		\begin{align}\label{eq:lem3.1_m98_8}
			I_{121}^{l}(r) \le C_T (\lVert v\rVert_{\bf{L}^{\alpha,T}} +1)^{\alpha}((T-t)^{-\frac{q}{2}}+1)^{1/q},
		\end{align}
		when $l$ is large. 
		
		For the term $I^{l}_{122}$ we again proceed as in the calculation \eqref{eq:lem3.1_m98_7}. Note that, as $\alpha \beta q < \alpha$ by our assumption, we can again apply Proposition \ref{prop:moment_est_Y_2} in the following. Let $r<t$.
		\begin{align}\label{eq:lem3.1_m98_9}
			I^{l}_{122}(r) = & 
			\E_Y \left(\langle Y_{(T-r)-}^{\alpha\beta}, v_l(r)^{\alpha}\rangle^q\right)^{1/q}
			\nonumber \\
			\le & \lVert v_l(r) \rVert_{\bf{L}^{\alpha}(\R)}^{\alpha}  \left[ \E_Y  \left( \frac{1}{ \lVert v_l(r) \rVert_{\bf{L}^{\alpha}(\R)}^{\alpha}}  \int_{\R} v_l(r,x)^{\alpha} Y_{T-r}(x)^{\alpha\beta q}\,d x \right) \right]^{1/q}
			\nonumber \\
			= &  \lVert v_l(r) \rVert_{\bf{L}^{\alpha}(\R)}^{\alpha-\alpha/q} \left[ \int_{\R} v_l(r,x)^{\alpha} \E_Y (Y_{T-r}(x)^{\alpha\beta q}) \,d x \right]^{1/q}
			\nonumber \\
			\le &  C_T \lVert v_l(r) \rVert_{\bf{L}^{\alpha}(\R)}^{\alpha-\alpha/q} \left[ \int_{\R} v_l(r,x)^{\alpha}(T-r)^{-\frac{\alpha\beta q}{2}} \,d x  + \int_{\R} v_l(r,x)^{\alpha} \,d x\right]^{1/q} 
			\nonumber \\
			\le & C_T   \lVert v_l(r) \rVert_{\bf{L}^{\alpha}(\R)}^{\alpha} ((T-t)^{-\frac{\alpha\beta q}{2}}+1)^{1/q}	\le C_T  (\lVert v\rVert_{\bf{L}^{\alpha,T}} +1)^{\alpha}((T-t)^{-\frac{\alpha\beta q}{2}}+1)^{1/q}
		\end{align} for large $l$. We can observe that \eqref{eq:lem3.1_m98_8} and \eqref{eq:lem3.1_m98_9} together show that $I^l_{12}\le C_{t,T}$ where $C_{t,T}$ is independent of $l$. Thus \textbf{(i)} is proved.

		\textbf{Proof of (ii): } First note that $v_l \to v$ in $\bf{L}^{\alpha, T}$ implies the following almost everywhere convergence along a sub-sequence: there exists a sequence $(l_i)_i$ of natural numbers such that $$v_{l_i}(r,x) \to  v(r,x) \text{ as } i \to \infty$$ for a.e. $(r,x) \in [0,T]\x\R$. We will abuse our notation again and use $l$ to denote this subsequence.
		
		By Proposition \ref{prop:moment_est_Y_2},
		\begin{align}\label{eq:lem3.1_m98_10}
			I^{l}_{2} = & \E_Y\left| \int_0^t  e^{-\langle Y_{(T-r)-}, v(r)\rangle} \left(b_n \langle Y_{(T-r)-}, v_{l}(r)^{\alpha}-v(r)^{\alpha}\rangle  
			- \langle Y_{(T-r)-}^{\alpha\beta}, v_{l}(r)^{\alpha}-v(r)^{\alpha}\rangle \right) \,d r \right|
			\nonumber \\
			\le &  \int_0^t \int_{\R} \left| v_{l}(r,x)^{\alpha} - v(r,x)^{\alpha}\right| \cdot \E_Y \left( b_n Y_{(T-r)-}(x)+ Y^{\alpha\beta}_{(T-r)-}(x)\right) \,d r \,d x
			\nonumber \\
			\le &  C_T \int_0^t \int_{\R} \left[ b_n (T-r)^{-\frac{1}{2}} + (T-r)^{-\frac{\alpha\beta}{2}} + 1 \right] \left| v_{l}(r,x)^{\alpha} - v(r,x)^{\alpha}\right| \,d r \,d x
			\nonumber \\
			\le & C_T (b_n (T-t)^{-\frac{1}{2}} + (T-t)^{-\frac{\alpha\beta}{2}} +1) \int_0^t \int_{\R}  \left| v_{l}(r,x)^{\alpha} - v(r,x)^{\alpha}\right| \,d r \,d x	
		\end{align}
		with $C$ being independent of $x$ and $l$.  The right hand side converges to $0$ as $l \to \infty$ as $v_l \to v$ in $\bf{L}^{\alpha, T}$. This proves \textbf{(ii)} .

	\end{proof}
	
	Next we prove Proposition \ref{prop:(3.19)_m98}. For the proof we will need to understand how $Y$ behaves when $Z^{(n)}$ jumps. Since $n$ is fixed in this proof, we drop it to simplify the notations introduced in Section \ref{sec:prelim}. We shall write $Z = Z^{(n)}$, $V = V^n$, $S = S^n$, $U = U^n$, $T_l= T^n_l$, $\tau = \tau^n$, $\mc{N} = \mc{N}^n$, $\hat{\mc{N}} = \hat{\mc{N}}^n$, $\gamma(s) : = \gamma^n(s)$ and $\gamma_l := \gamma^n(T^n_l)$ for $l \in \N$. Also recall the notation  
	\begin{align*}
		\theta(s,x, \lambda) := \theta_n(s,x, \lambda) = e^{-\langle Y_{T-\gamma^n(s)} , Z^{(n)}_{\gamma^n(s) -}\rangle} \left( e^{- \lambda Y_{T-\gamma^n(s)}(x) } -1 \right).
	\end{align*}
	
	\begin{proof}[Proof of Proposition \ref{prop:(3.19)_m98}]	
		
		Fix $t \in [0,T)$ and let 
		\begin{align*}
			\theta_j := \theta(T_j, U_j, S_j) = e^{ -\langle Y_{T-\gamma_j}, Z_{\gamma_j} \rangle} - e^{ -\langle Y_{T-\gamma_j}, Z_{\gamma_j-} \rangle} 
			=e^{ -\langle Y_{T-\gamma_j}, Z_{\gamma_j-} \rangle} \left(e^{-S_j Y_{T-\gamma_j}(U_j)} -1\right).
		\end{align*} Suppose we show that on the event $ \{\gamma_l \le t < \gamma_{l+1}\}$ we have,		
		\begin{align}\label{eq:eq(3.19)_m98_7}
			\E_Y \left[e^{-\langle Y_{T-t}, Z_t\rangle}\right]			
			=   \E_Y\left[e^{-\langle Y_T, Z_0 \rangle} + \int_{0}^{t}\tilde{\mc{I}}(Y_{(T-r)-}, Z_{r-} ) \,d r 
			+	\sum_{i=1}^l\theta_i \right], 			
		\end{align}
		then we can write
		\begin{align*}
			\sum_{i=1}^l \theta_i = \int_0^{\tau(t)} \int_{\R} \int_{\R_+} \theta (s, x, \lambda) \mc{N}(\,d \lambda, \, d x, \, d s),
		\end{align*}
		since for $\gamma_l \le t < \gamma_{l+1}$ by definition (see \eqref{eq:gamma^Z,n} and  \eqref{eq:tau^n}) $\tau(t) \in [ T_l , T_{l+1})$. Replace the above in \eqref{eq:eq(3.19)_m98_7} and we obtain \eqref{eq:(3.19)_m98}. So, to complete the proof of \eqref{eq:(3.19)_m98} we need to establish \eqref{eq:eq(3.19)_m98_7}.	
		
		We will prove this by induction on $l = 0, 1, 2 , \ldots$ and use \eqref{eq:(3.5)_m98} repeatedly in the following. Note that \eqref{eq:eq(3.19)_m98_7} for $t=0$ is trivial. When $0= \gamma_0 < t < \gamma_1$, by our convention $l=0$. In this case \eqref{eq:eq(3.19)_m98_7} is 
		\begin{align}\label{eq:eq(3.19)_m98_8}
			\E_Y \left[e^{-\langle Y_{T-t}, Z_t\rangle}	\right]		
			=  \E_Y\left[e^{-\langle Y_T, Z_0 \rangle} + \int_{0}^{t}\tilde{\mc{I}}(Y_{(T-r)-}, Z_{r-} ) \,d r  \right] 			
		\end{align}
		and this follows directly from \eqref{eq:(3.5)_m98}. 
		
		Now assume that \eqref{eq:eq(3.19)_m98_7} holds on the event $\{\gamma_l \le s < \gamma_{l+1}\}$. We first show that 
		\begin{align}\label{eq:eq(3.19)_m98_9}
			& \E_Y\left[ e^{-\langle Y_{T-\gamma_{l+1}}, Z_{\gamma_{l+1}}\rangle}\right]
			\nonumber \\
			= &  \E_Y\left[e^{-\langle Y_T, Z_0 \rangle} + \int_{0}^{\gamma_{l+1}}\tilde{\mc{I}}(Y_{(T-r)-}, Z_{r-} ) \,d r 
			+	\sum_{i=1}^{l+1}\theta_i \right].					
		\end{align}
		
		By definition of $\theta_{l+1}$ and induction hypothesis,
		\begin{align*}
			& \E_Y \left[e^{-\langle Y_{T-\gamma_{l+1}}, Z_{\gamma_{l+1}}\rangle}\right]
			=  \E_Y [\theta_{l+1} ]+ \E_Y \left[e^{-\langle Y_{T-\gamma_{l+1}}, Z_{\gamma_{l+1}-}\rangle}\right]
			\\
			= & \E_Y [\theta_{l+1} ] + \lim_{\substack{{s\uparrow \gamma_{l+1}}\\{\gamma_l \le s<\gamma_{l+1}} } } 
			\E_Y e^{-\langle Y_{T-s}, Z_{s}\rangle}
			\\
			= & \E_Y [\theta_{l+1} ] + \lim_{\substack{{s\uparrow \gamma_{l+1}}\\{\gamma_l \le s<\gamma_{l+1}} } }  \E_Y\left[e^{-\langle Y_T, Z_0 \rangle} + \int_{0}^{s}\tilde{\mc{I}}(Y_{(T-r)-}, Z_{r-} ) \,d r 
			+	\sum_{i=1}^l\theta_i \right] 
			\\
			= & \E_Y [\theta_{l+1} ] +  \E_Y\left[e^{-\langle Y_T, Z_0 \rangle} + \int_{0}^{\gamma_{l+1}}\tilde{\mc{I}}(Y_{(T-r)-}, Z_{r-} ) \,d r 
			+	\sum_{i=1}^l\theta_i	 \right] 
			\\			
			= &  \E_Y\left[e^{-\langle Y_T, Z_0 \rangle} + \int_{0}^{\gamma_{l+1}}\tilde{\mc{I}}(Y_{(T-r)-}, Z_{r-} ) \,d r 
			+	\sum_{i=1}^{l+1}\theta_i \right]. 			
		\end{align*}
		This proves \eqref{eq:eq(3.19)_m98_9}.
		
		The last step of the induction is to prove \eqref{eq:eq(3.19)_m98_7} when $l$ is replaced with $l+1$ and $\gamma_{l+1} < t < \gamma_{l+2}$. We use \eqref{eq:(3.5)_m98} with $T - \gamma_{l+1}$, $t -\gamma_{l+1}$ instead of $T$, $t$ and then apply  \eqref{eq:eq(3.19)_m98_9} to get,
		\begin{align}\label{eq:eq(3.19)_m98_10}
			&\E_Y \left[e^{-\langle Y_{T-t}, Z_t\rangle}\right] = \E_Y\left[ \exp\left(- \langle Y_{T-t}, V_{t-\gamma_{l+1}}(Z_{\gamma_{l+1}})\rangle\right)\right]
			\nonumber \\			
			= & \E_Y \left[\exp\left( -\langle Y_{T - \gamma_{l+1}  }, V_{0}(Z_{\gamma_{l+1}})\rangle \right)\right]
			+ \E_Y \left[\int_{0 }^{t-\gamma_{l+1 }} \tilde{\mc{I}} (Y_{(T-\gamma_{l+1} -r)-}, V_r(Z_{\gamma_{l+1}})) \,d r	\right]
			\nonumber \\
			= & \E_Y\left[e^{-\langle Y_T, Z_0 \rangle} + \int_{0}^{\gamma_{l+1}}\tilde{\mc{I}}(Y_{(T-r)-}, Z_{r-} ) \,d r 
			+	\sum_{i=1}^{l+1}\theta_i \right] 
			\nonumber \\
			& +  \E_Y\left[ \int_{\gamma_{l+1}}^{t} \tilde{\mc{I}} (Y_{(T -r)-}, Z_{r-}) \,d r\right],			 
		\end{align}
		which is the required expression. This completes the induction argument and proves \eqref{eq:eq(3.19)_m98_7}.		
		
	\end{proof}
	
	For the proof of our final proposition, we continue to suppress $n$ and use the notations introduced before the previous proof. Define
	\begin{align}\label{eq:(3.20)_m98_mart}
		M_s = \int^s_0 \int_{\R} \int_{0}^{\infty} \theta(r,x,\lambda) [\mc{N}(d \lambda, \, d x, \, d r ) - \hat{\mc{N}}(d \lambda, \, d x, \, d r)]
	\end{align} and note that $M$ is an $\mc{F}^{Z^{(n)}}$-martingale. 
	
	\begin{proof}[Proof of Proposition \ref{prop:lem3.3_m98}]
		We recall that
		\begin{align*}
			\eta = 	 \frac{\alpha \beta(\alpha \beta -1)}{\Gamma (2-\alpha \beta)} \text { and }g(r,y)  = \int_{0+}^r (e^{-\lambda y} -1 +\lambda y)\lambda^{-\alpha\beta-1} \,d \lambda, \quad r, y \ge 0.
		\end{align*} Since for all $y\ge 0$, 
		\begin{align*}
			y^{\alpha\beta}  = \eta \int_{0+}^{\infty}(e^{-\lambda y} -1 +\lambda y)\lambda^{-\alpha\beta -1} \,d \lambda
			= \eta g(1/n, y) + \eta \int_{1/n}^{\infty} (e^{-\lambda y} -1 )\lambda^{-\alpha\beta -1} \,d \lambda + b_n y,
		\end{align*}		
		we can write
		\begin{align}\label{eq:(3.20)_m98_2}
			& \E_Y\left[\int_{0}^{t}\tilde{\mc{I}}(Y_{(T-r)-}, Z_{r-} ) \,d r \right]
			=  \E_Y \left[\int^t_{0} e^{-\langle Y_{(T-r)-}, Z_{r-}\rangle}  \langle b_n Y_{(T-r)-} -Y^{\alpha\beta}_{(T-r)-}, Z_{r-}^{\alpha} \rangle  \,d r\right]
			\nonumber \\
			= & -  \eta  \E_Y \left[\int^t_{0} e^{-\langle Y_{(T-r)-}, Z_{r-}\rangle} \langle g(1/n, Y_{(T-r)-}(\cdot)), Z_{r-}^{\alpha} \rangle  \,d r\right]
			\nonumber \\
			& -  \eta  \E_Y\left[ \int^t_{0} e^{-\langle Y_{(T-r)-}, Z_{r-}\rangle} \langle \int_{1/n}^{\infty} (e^{-\lambda Y_{(T-r)-}(\cdot)} -1 )\lambda^{-\alpha\beta -1} \,d \lambda, Z_{r-}^{\alpha} \rangle  \,d r\right].		
		\end{align}
		
		Let 
		\begin{align*}		
			h(r)  = &  e^{-\langle Y_{T-r}, Z_{r-} \rangle} \int_{\R} \frac{\left(Z_{r-}(x)\right)^{\alpha}}{\lVert Z_{r-} \rVert_{\alpha}^{\alpha} } \int_{1/n}^{\infty} \left( e^{-\lambda Y_{T-r}(x)} -1 \right) \lambda^{-\alpha\beta -1} \,d \lambda \,d x 
			\\
			\beta(r)  = &\lVert Z_{r-}\rVert^{\alpha}_{\alpha}.
		\end{align*}		
		
		Then by definition $\gamma(t) = \inf\{ s\ge 0 \mid \int_0^s \beta(s) \,d s > t\}$ and also recall from \eqref{eq:tau^n} that $\gamma (\tau(s)) = s$. Given $Y$, applying \cite[Exercise 6.12]{ek}, for any $s \ge 0$,  we have
		\begin{align}\label{eq:(3.20)_m98_4}
			& \eta  \int^{s}_{0} e^{-\langle Y_{(T-r)-}, Z_{r-}\rangle} \langle \int_{1/n}^{\infty} (e^{-\lambda Y_{(T-r)-}(\cdot)} -1 )\lambda^{-\alpha\beta -1} \,d \lambda, Z_{r-}^{\alpha} \rangle  \,d r
			\nonumber \\
			= & \eta \int^{s}_{0} h(r)\beta(r)\,d r = \eta \int^{\gamma(\tau(s))}_{0} h(r)\beta(r)\,d r =  \eta \int^{\tau(s)}_{0} h(\gamma(r))\,d r
			\nonumber \\
			= & \eta \int^{\tau(s)}_{0} e^{-\langle Y_{T-\gamma(r)}, Z_{\gamma(r)-} \rangle} \int_{\R} \frac{\left(Z_{\gamma(r)-}(x)\right)^{\alpha}}{\lVert Z_{\gamma(r)-} \rVert_{\alpha}^{\alpha} } \int_{1/n}^{\infty} \left( e^{-\lambda Y_{T-\gamma(r)}(x)} -1 \right) \lambda^{-\alpha\beta -1} \,d \lambda \, d x \, d r
			\nonumber \\
			= & \eta \int^{\tau(s)}_{0} \int_{\R} \int_{0}^{\infty} \theta(r,x,\lambda) \frac{\left(Z_{\gamma(r)-}(x)\right)^{\alpha}}{\lVert Z_{\gamma(r)-} \rVert_{\alpha}^{\alpha} } \ind(\lambda > 1/n)\lambda^{-\alpha\beta -1} d \lambda \,d x \,d r
			\nonumber \\
			= &  \int^{\tau(s)}_{0} \int_{\R} \int_{0}^{\infty} \theta(r,x,\lambda) \hat{\mc{N}} (d \lambda, \, d x, \, d r)
		\end{align}	
		using Lemma \ref{lem:lem3.2_m98} in the last line. 
		
		Combining \eqref{eq:(3.20)_m98_mart} and the calculations in \eqref{eq:(3.20)_m98_2}, \eqref{eq:(3.20)_m98_4} we get,
		\begin{align}\label{eq:(3.20)_m98_6}
			& \E_Y \left[\int_{0}^{t}\tilde{\mc{I}}(Y_{(T-r)-}, Z_{r-} ) \,d r 
			+ \int_{0}^{\tau(t)} \int_{\R} \int_{\R_+} \theta(s,x,\lambda) \mc{N}(d \lambda, \, d x, \, d s)  \right]
			\nonumber \\
			= & \E_Y \left[ M_{\tau(t)} - \eta  \int_{0}^{t} e^{-\langle Y_{(T-r)-}, Z_{r-}\rangle} \langle g(1/n, Y_{(T-r)-}(\cdot)), Z_{r-}^{\alpha} \rangle  \,d r  \right]
		\end{align} 
		We can now use \eqref{eq:(3.20)_m98_6} to rewrite \eqref{eq:(3.19)_m98}. 
		\begin{align}\label{eq:(3.21)_m98_1}
			& \E_Y \left[e^{-\langle Y_{T-t}, Z_t\rangle} \right]
			\nonumber \\		
			& = \E_Y \left[M_{\tau(t)}\right] - \E_Y \left[ \eta  \int_{0}^{t} e^{-\langle Y_{(T-r)-}, Z_{r-}\rangle} \langle g(1/n, Y_{(T-r)-}(\cdot)), Z_{r-}^{\alpha} \rangle  \,d r  \right]		
		\end{align} 		
		Recall the notation $\Upsilon_m(t) = \gamma(m)\wedge t$ and observe that for any $m \in \N$, $\tau(\Upsilon_m(t)) = \tau(t)\wedge m$.  We localize the above as follows.		
		\begin{align}
			& \E_Y \left[e^{-\langle Y_{T-\Upsilon_m(t)}, Z_{\Upsilon_m(t)}\rangle}\right] 
			\nonumber \\		
			= & \E_Y\left[ e^{-\langle Y_T, Z_0 \rangle} \right] + \E_Y\left[ M_{\tau(t)\wedge m}	\right]
			\nonumber \\
			& - \eta  \E_Y \left[ \int^{\Upsilon_m(t)}_{0} e^{-\langle Y_{(T-r)-}, Z_{r-}\rangle} \langle g(1/n, Y_{(T-r)-}(\cdot)), Z_{r-}^{\alpha} \rangle  \,d r\right]	
		\end{align}
		Apply $\E_Z$ to the above. As $$\E_Z \E_Y( M_{ \tau^Z(t)\wedge m} ) = \E_Y \E_Z( M_{ \tau^Z(t)\wedge m} ) = 0$$ we have,
		\begin{align}
			\E \left[e^{-\langle Y_{T-\Upsilon_m(t)}, Z_{\Upsilon_m(t)}\rangle}\right] 
			=& \E\left[ e^{-\langle Y_T, Z_0 \rangle}\right] 	
			\nonumber \\
			& - \eta  \E \left[ \int^{\Upsilon_m(t)}_{0} e^{-\langle Y_{(T-r)-}, Z_{r-}\rangle} \langle g(1/n, Y_{(T-r)-}(\cdot)), Z_{r-}^{\alpha} \rangle  \,d r \right].
		\end{align}
		This is the required expression.
	\end{proof}	
	
	\appendix
	\section{A Gr\"{o}nwall-type Lemma}\label{app:gronwall} 
		
		We first state the ordinary Gr\"{o}nwall lemma.		
		\begin{lem}\label{lem:gronwall_0}
			Let $T>0$ and $f$, $g$ and $h$ be non-negative integrable functions on $[0,T]$ satisfying the following inequality for all $t \in [0,T]$,
			\begin{align}
				f(t) \le g(t) + \int_0^t h(s)f(s)\, ds. 
			\end{align}
			Then for a.e. $t \in [0,T]$ we have,
			\begin{align}
				f(t) \le g(t) + \int_0^t g(s) h(s) \exp \left( \int_0^s h(r) \, dr \right) \, ds.
			\end{align}
		\end{lem}		
		The proof is omitted as it is standard. We can now use the above result to prove the required estimate. 			
		
		\begin{lem}\label{lem:gronwall}
			Let $\gamma, \theta \in (0,1)$ and $f:(0, T] \to [0, \infty)$ be an integrable function such that and for all $t \in [0,T]$
			\begin{align}\label{eq:gronwall_hyp}
				f(t) \le ct^{-\theta} + c\int_0^t (t-r)^{-\gamma}f(r) \, d r
			\end{align}
			for some constant $c>0$. Then there exists an integrable function $C_1: (0,T] \to [0, \infty)$ and a constant $C_2 >0$ such that, for a.e. $t \in [0,T]$,
			\begin{align}
				f(t) \le C_1(t) + \int_0^t C_1(s) \exp(C_2 s) \,d s.
			\end{align}
			Moreover $C_1, C_2$ are independent of the function $f$.
		\end{lem}
		
		\begin{proof}
			Let $k>0$ be the smallest integer such that $\gamma < \frac{k}{k+1}$ and $t \in [0,T]$. We apply \eqref{eq:gronwall_hyp} and use the substitutions $w = \frac{r}{t}$ and $v = \frac{r-u}{t-u}$ for the first and second integrals in the RHS of the following computation.
			\begin{align}\label{eq:gronwall_1}
				f(t) \le &  c t^{-\theta} + c^2 \int_0^t (t-r)^{-\gamma} t^{-\theta} \,d r + c^2 \int_0^t \int_0^r (t-r)^{-\gamma}(r-u)^{-\gamma} f(u)\, du \, dr
				\nonumber \\
				= & c t^{-\theta} + c^2 t^{1-\gamma - \theta} \int_0^1 (1-w)^{-\gamma} w^{- \theta} \,  dw +c^2 \int_0^t \,d u f(u) \int_u^t \,d r (t-r)^{-\gamma}(r-u)^{-\gamma} 
				\nonumber \\
				= & c_1(t) + c^2 \int_0^t f(u) (t-u)^{1-2\gamma} \, d u \int_0^1 \,d v   (1-v)^{-\gamma}v^{-\gamma}
				\nonumber \\
				= & c_1(t) + c'_1 \int_0^t f(u) (t-u)^{1-2\gamma} \, d u		
			\end{align}		
			where $c_1(t) = ct^{-\theta} + c^2 B(1-\gamma, 1- \theta)\, t^{1-\gamma - \theta} $ and $c'_1 = c^2 B(1-\gamma, 1-\gamma)$ with $B$ here denoting the Beta function. Again applying \eqref{eq:gronwall_hyp} to \eqref{eq:gronwall_1} we have
			\begin{align}	
				f(t) \le & c_2(t) + c'_2\int_0^t f(u) (t-u)^{2-3\gamma} \, d u,
				\nonumber 
			\end{align}		
			where $c_2(t) = c_1(t) + c'_1 c  B(2-2\gamma, 1- \theta) \, t^{2-2\gamma - \theta} $ and $c'_2 = c'_1 c B(1-\gamma, 2-2\gamma)$. Continuing	this process for $k$ steps we get,
			\begin{align}\label{eq:gronwall_2}
				f(t)  \le  & c_k(t) + c'_k \int_0^t f(u)(t-u)^{k - (k+1)\gamma} \, d u
				\nonumber \\
				\le & c_k(t) + c'_k T^{k - (k+1)\gamma} \int_0^t f(u) \, du.
			\end{align}	
			where the last step is obtained by our assumption on $k$. Also note that $f$ is non-negative and integrable on $[0,T]$ by hypothesis. Therefore we can apply the standard Gr\"{o}nwall's inequality from Lemma \ref{lem:gronwall_0} and have,
			\begin{align}
				f(t) \le c_k(t) + \int_0^t  c_k(s) \exp(c'_k T^{k - (k+1)\gamma} s) \, ds
				\nonumber			 
			\end{align}
			for a.e. $t \in [0,T]$.	We can thus define $C_1(t) = c_k(t)$ and $C_2 = c'_k  T^{k - (k+1)\gamma}$. Clearly these are independent of $f$. To see that $C_1$ is integrable on $(0, T]$ we only note that each $t^{m-m\gamma - \theta}$ ($m = 0 ,\ldots, k$) is integrable. 
		\end{proof}
		
		\section{Norm Estimates for Solutions of the Evolution Equation}\label{app:pde}
		This section contains some useful properties of the solutions to the PDE
		\begin{align}\label{eq:pde_Z_2}
			\frac{\partial}{\partial t} v(t,x) & = \frac{1}{2}\frac{\partial^2}{\partial x^2}v(t,x) - b_n v(t,x)^{\alpha}, \quad x \in \R, t \in [0,T], 
			\nonumber \\
			v(0, \cdot) & = \varphi.
		\end{align}
		where $T$ is arbitrary but finite. When $\varphi \in \mc{S}(\R)_+$, \cite[Theorem A]{iscoe86} guarantees that this equation admits a unique solution.
		
		\begin{lem}\label{lem:pde_Z_properties}
			If $v = v(t,x)$ solves the PDE \eqref{eq:pde_Z_2} and $\varphi \in \mc{S}(\R)_+$, then $v$ satisfies all the hypotheses of Proposition \ref{prop:time_dep_mart_prob}.
		\end{lem}
		
		\begin{proof}
			\textbf{(a)}  It follows from \cite[Proposition A2]{fleisch86} and proof of \cite[Lemma 2.1(c)]{mytnik02} that $s \mapsto v(s) \in {\bf{L}}^{\eta}(\R)\cap {\bf{L}}^{\rho}(\R)$ is continuous.
			
			\textbf{(b)} 
			We first prove that 
			\begin{align}\label{eq:pde_Z_prop_1}
				\sup_{s\le T}\left\lVert \frac{\partial^2}{\partial x^2}v(s) \right\rVert_{\infty} =\left\lVert \frac{\partial^2}{\partial x^2}v\right\rVert_{{\bf{L}}^{\infty}([0,T]\x \R)}  < \infty.
			\end{align}
			Note that as $\varphi \in \mc{S}_+$, by \cite[Theorem A]{iscoe86}, the solution $v : [0, T] \to C_0(\R_+)_+$ (continuous, non-negative functions vanishing at infinity) is a continuous map. Therefore by \eqref{eq:pde_Z_2}, to show \eqref{eq:pde_Z_prop_1} it is enough to prove that
			\begin{align}\label{eq:pde_Z_prop_1.5}
				\sup_{s\le t}\lVert w(s)\rVert_{\infty}<\infty
			\end{align}
			where we have used the notation $w(s) = \dot{v}(s) = \frac{\partial}{\partial s}v(s).$		
			
			From the proof of \cite[Theorem A]{iscoe86} it follows that $w$ must satisfy the PDE
			\begin{align}\label{eq:pde_Z_derivative}
				w(t) = P_t (\tilde{\varphi}) -  \alpha b_n \int_0^t P_{t-s}(v(s)^{\alpha-1}w(s)) \, d s
			\end{align}
			where $\tilde{\varphi} = \frac{1}{2}\frac{\partial^2}{\partial x^2}\varphi - \frac{b_n}{2} \varphi^{\alpha}$. This gives us,
			\begin{align}\label{eq:pde_Z_prop_2}
				\lVert w(t) \rVert_{\infty} \le \lVert \tilde{\varphi}\rVert_{\infty} + \alpha b_n \int_0^t \lVert (v(s)^{\alpha-1}\rVert_{\infty} \lVert w(s)\rVert_{\infty} \,d s,
			\end{align}
			from which using Gr\"{o}nwall's inequality (see \cite[Appendix B2]{evans10}) we obtain
			\begin{align}
				\lVert w(t) \rVert_{\infty}  \le  \lVert \tilde{\varphi}\rVert_{\infty} \exp \left( \alpha b_n  \int_0^t \lVert v(s)^{\alpha-1}\rVert_{\infty} \,d s \right) < \infty.
			\end{align}	
			As $v$ is continuously differentiable (see the proof of \cite[Theorem A]{iscoe86}),  $s\mapsto w(s, \cdot)$ is continuous. This fact along with the above gives us \eqref{eq:pde_Z_prop_1.5}.
			
			Next we show that the map $[0,T] \to \bf{L}^{\infty}(\R)$, $t \mapsto \frac{\partial^2}{\partial x^2} v(t)$ is continuous, i.e.
			\begin{align}\label{eq:pde_Z_prop_3}
				\left\lVert \frac{\partial^2}{\partial x^2} v(s) - \frac{\partial^2}{\partial x^2} v(t) \right\rVert_{\infty} \to 0 \text{ as } s \to t \text{ in } [0,T].
			\end{align}
			Similarly as above, since $v \in C_0(\R)_+$, by \eqref{eq:pde_Z_2} it is enough to show that 
			\begin{align} \label{eq:pde_Z_prop_4}
				\lVert w(t) - w(s)\rVert_{\infty} \to 0 \text{ as }s \to t			
			\end{align}
			and we use \eqref{eq:pde_Z_derivative} for this purpose.
			
			Let $0\le s < t \le T$. Let $f_r = v(r)^{\alpha-1}w(r)$. Then from \eqref{eq:pde_Z_derivative}
			\begin{align}\label{eq:pde_Z_prop_5}
				w(t) - w(s) = & P_t(\tilde{\varphi}) - P_s(\tilde{\varphi}) - \alpha b_n \left[ \int_0^t P_{t-r} (f_r) \, dr - \int_0^s P_{s-r} (f_r)  \, dr\right]
				\nonumber \\
				= & P_t(\tilde{\varphi}) - P_s(\tilde{\varphi}) - \alpha b_n \left[ \int_s^t P_{t-r} (f_r) \, dr + \int_0^s \left( P_{t-r} (f_r)  - P_{s-r} (f_r) \right)  \, dr\right]
				\nonumber \\
				= & P_t(\tilde{\varphi}) - P_s(\tilde{\varphi}) - \alpha b_n \left[ \int_s^t P_{t-r} (f_r) \, dr + \int_0^s \left( P_{t-s}( P_{s-r}(f_r))  - P_{s-r} (f_r) \right)  \, dr\right]
			\end{align}
			using that fact $P_{t-r} f_r = P_{t-s} (P_{s-r} f_r)$. 
			
			As $\varphi \in \mc{S}(\R)_+$, by our definition $\tilde{\varphi} \in \mc{S}(\R)$. Therefore when $s \to t$, $\lVert P_t \tilde{\varphi} - P_s \tilde{\varphi} \rVert_{\infty} \to 0$. For the second term in \eqref{eq:pde_Z_prop_5}, note that if we can prove that $$\sup_{r \le t} \lVert P_{t-r}(f_r) \rVert_{\infty} <\infty, $$ it will follow that $ \int_s^t P_{t-r} (f_r) \,d r \to 0$ in $\bf{L}^{\infty}(\R)$ as $s \to t$. We have, for  $x \in \R$
			\begin{align*}
				\left| P_{t-r}(f_r) (x)\right| \le \left| \int_{\R} p_{t-r}(x-y) f_r(y) \,d r \right| \le \lVert f_r \rVert_{\infty} = \lVert v(r)^{\alpha -1} w(r) \rVert_{\infty} <\infty.
			\end{align*}	
			since know $v (r) \in C_0(\R)$ and we have already shown that $\sup_{r \le t} \lVert w(r)\rVert_{\infty} <\infty$. Similarly, the third term in \eqref{eq:pde_Z_prop_5} can be shown to be converging to $0$ in $\bf{L}^{\infty}(\R)$ as $s \to t$. This proves \eqref{eq:pde_Z_prop_4} and hence \eqref{eq:pde_Z_prop_3}. 		
			
			\textbf{(c)} Let $\sigma = \frac{\alpha\beta}{\alpha\beta -1}$. To show $ \sup_{s\le t}\lVert \dot{v}(s)\rVert_{\sigma}<\infty $ we again use \eqref{eq:pde_Z_derivative}. Note that
			\begin{align}\label{eq:w_Lp}
				\int_{\R} |w(t,x)|^{\sigma}\,d x \le C \left[ \int_{\R} |P_t (\tilde{\varphi})|^{\sigma} \,d x +  \alpha b_n \int_{\R} \left| \int_0^t P_{t-s}(v(s)^{\alpha-1}w(s))\,d s \right|^{\sigma} \,d x \right].
			\end{align}
			Using Jensen inequality,
			\begin{align*}
				\int_{\R} |P_t (\tilde{\varphi})|^{\sigma} \,d x =  &  \int_{\R} \left| \int_{\R} p_t (x-y ) \tilde{\varphi} (y) \,d y \right|^{\sigma} \,d x
				\\
				\le & \int_{\R} \int_{\R} p_t (x-y ) |\tilde{\varphi} (y)|^{\sigma} \,d y \,d x
				\\
				= & \lVert \tilde{\varphi} \rVert_{\infty}^{\sigma} < \infty.
			\end{align*}
			By definition of $P_{t-s}$ and using Jensen's inequality once more we have, 
			\begin{align*}
				\int_{\R} \left| \int_0^t P_{t-s}(v(s)^{\alpha-1}w(s))\,d s \right|^{\sigma} \,d x = &  
				\int_{\R} \left| \int_0^t \int_{\R} p_{t-s}(x-y) v(s,y)^{\alpha-1}w(s,y) \,d y \,d s \right|^{\sigma} \,d x
				\\
				\le & t^{\sigma-1} \int_{\R}  \int_0^t \int_{\R} p_{t-s}(x-y) |v(s,y)^{\alpha-1}w(s,y)| ^{\sigma}  \,d y \,d s\,d x.				
			\end{align*}
			Using Fubini's theorem, as all terms are non-negative, integrating out $x$ in the above we have 
			\begin{align}\label{eq:pde_Z_properties_2}
				\int_{\R} \left| \int_0^t P_{t-s}(v(s)^{\alpha-1}w(s))\,d s \right|^{\sigma} \,d x			
				& \le  t^{\sigma-1} \int_0^t \int_{\R} |v(s,y)^{\alpha-1}w(s,y)| ^{\sigma}  \,d y \,d s	
				\nonumber \\		
				& \le  C_1 \int_0^t \lVert v(s) \rVert_{\infty}^{\sigma(\alpha-1)} \lVert w(s) \rVert_{\sigma}^{\sigma} \, d s.
			\end{align}
			
			Using \eqref{eq:pde_Z_properties_2} in \eqref{eq:w_Lp} we have
			\begin{align}
				\lVert w(t) \rVert_{\sigma}^{\sigma} \le C \lVert \tilde{\varphi} \rVert_{\infty}^{\sigma} + C \int_0^t \lVert v(s) \rVert_{\infty}^{\sigma(\alpha-1)} \lVert w(s) \rVert_{\sigma}^{\sigma} \, d s. 
			\end{align}
			Again by Gr\"{o}nwall's inequality
			\begin{align}
				\lVert w(t) \rVert_{\sigma}^{\sigma}  \le C \lVert \tilde{\varphi} \rVert_{\infty}^{\sigma} \exp \left( \int_0^t \lVert v(s) \rVert_{\infty}^{\sigma(\alpha-1)}  \, d s\right) < \infty
			\end{align}
			and the required result follows as in the previous part.
		\end{proof}	
		
		\section{Proof of Proposition \ref{prop:time_dep_mart_prob}}\label{app:time_dep_mart}
		
		Since the proof is a little long we carry it out in two steps. The first shows that a solution of the weak form \eqref{eq:main_spde_weak_form} also satisfies a time-dependent version as described in \eqref{eq:time_dep_weak_form_0}. The proof follows the argument of \cite[Theorem 2.1]{shiga94}.
		
		\begin{lem}\label{lem:time_dep_weak_form}
			Let $T>0$ be fixed and assume that $Y$ satisfies \eqref{eq:main_spde_weak_form} and the following conditions hold for $\psi :[0,T]\times \R \to [0,\infty)$.
			\begin{itemize}
				\item[(i)]  The map $[0, T] \ni s \mapsto \psi_s  \in \bf{L}^{\eta}(\R)\cap \bf{L}^{\rho}(\R)$ is continuous, for some fixed $\eta \in (\frac{1}{\beta}, \alpha)$ and $\rho \in (\alpha, \frac{\alpha}{\beta}\wedge 2 )$.	
				
				\item[(ii)] $	\sup_{s \le T} \lVert \frac{\partial}{\partial s}\psi_s\rVert_{\frac{\alpha\beta}{\alpha\beta-1}} < \infty$, and
				
				\item[(iii)]  $s \mapsto \frac{\partial^2}{\partial x^2} \psi_s$ is continuous in $\bf{L}^{\infty}(\R)$, i.e. $\lVert \frac{\partial^2}{\partial x^2} \psi_s - \frac{\partial^2}{\partial x^2} \psi_t \rVert_{\infty} \to 0 $	as $|s-t|\to 0$.
			\end{itemize}
			Then for each $t \in [0,T]$, we have
			\begin{align}\label{eq:time_dep_weak_form_0}
				\langle  Y_t, \psi_t \rangle  =  \langle  Y_0, \psi_0  \rangle  + \int_0^t \langle  Y_s, \left(\frac{1}{2}\frac{\partial^2}{\partial x^2}  +\frac{\partial}{\partial s}\right)\psi_s  \rangle \,d s  +  \int_0^t \int_{\R} (Y_{s-}(x))^{\beta} \psi_s(x)  L^{\alpha} (d x, \, d s).
			\end{align}
		\end{lem}	
		
		\begin{proof}
			Fix $0\le t \le T$ and let $\Delta = \{0=t_0 < t_1 < \cdots < t_N = t\}$ be a partition of $[0,t]$. For all $s \in [t_{i-1},t_i]$, denote $\pi_{\Delta}(s) = t_{i-1}$ and $\bar{\pi}_{\Delta}(s)= t_i$. Then we have
			\begin{align}\label{eq:time_dep_weak_form_1}
				& \langle Y_t, \psi_t \rangle - \langle Y_0, \psi_0 \rangle
				\nonumber \\
				= & \sum_{i=1}^{N}  (\langle Y_{t_i}, \psi_{t_i} - \psi_{t_{i-1}} \rangle - \langle Y_{t_i}- Y_{t_{i-1}}, \psi_{t_{i-1} } \rangle ) 
				\nonumber \\
				= & \sum_{i=1}^N \left[ \int_{t_{i-1}}^{t_i} \langle Y_{\bar{\pi}_{\Delta}(s)}, \dot{\psi}_s \rangle \,d s  + \int_{t_{i-1}}^{t_i} \langle Y_s, \frac{1}{2}\frac{\partial^2}{\partial x^2}\psi_{\pi_{\Delta}(s)} \rangle \,d s + \int_{t_{i-1}}^{t_i} \int_{\R} Y_{s-}(x)^{\beta}\psi_{\pi_{\Delta}(s)}(x)L^{\alpha}(d x, \, d s)\right] 
				\nonumber \\
				= & \int_0^t \left( \langle Y_{\bar{\pi}_{\Delta}(s)}, \dot{\psi}_s \rangle  +  \langle Y_s, \frac{1}{2}\frac{\partial^2}{\partial x^2}\psi_{\pi_{\Delta}(s)} \rangle \right) \,d s + \int_0^t \int_{\R} Y_{s-}(x)^{\beta}\psi_{\pi_{\Delta}(s)}(x)L^{\alpha}(d x, \, d s)
			\end{align}
			
			To prove the lemma we have to show that, as $|\Delta| \to 0$
			\begin{itemize}
				\item[(a)]  $\int_0^t \langle Y_{\bar{\pi}_{\Delta}(s)}, \dot{\psi}_s \rangle \,d s \to \int_0^t \langle Y_s, \dot{\psi}_s \rangle \,d s$ a.s. , 
				
				\item[(b)] $ \int_0^t \langle Y_s, \frac{1}{2}\frac{\partial^2}{\partial x^2}\psi_{\pi_{\Delta}(s)} \rangle  \,d s \to \int_0^t \langle  Y_s, \frac{1}{2}\frac{\partial^2}{\partial x^2} \psi_s  \rangle \,d s$ a.s. , and
				
				\item[(c)] $\int_0^t \int_{\R} Y_{s-}(x)^{\beta}\psi_{\pi_{\Delta}(s)}(x)L^{\alpha}(\,d s, \,d x) \to \int_0^t \int_{\R^d} (Y_{s-}(x))^{\beta} \psi_s(x)  L^{\alpha} (d x, \, d s)$ in probability.
			\end{itemize}
			
			For (a) and (b), we need to show that the integrand converges pointwise (i.e. for each $s$) and that the dominated convergence theorem (DCT) can be applied. 
			
			(a)  Recall that $s \mapsto Y_s$ is right continuous measure-valued a.s. and by the definition of weak convergence we have, for each $s\in [0,t]$  
			\begin{align*}
				|\langle Y_{\bar{\pi}_{\Delta}(s)} - Y_s, \dot{\psi}_s \rangle|  \to 0, \text{	a.s.}
			\end{align*}
			as $\dot{\psi}_s$ is bounded and continuous (in the space variable). 			
			
			By H\"{o}lder's inequality, as $\alpha\beta >1$, we have a.s.
			\begin{align*}
				|\langle Y_{\bar{\pi}_{\Delta}(s)} - Y_s, \dot{\psi}_s \rangle| 
				\le & \int_{\R} \lvert Y_s(x) \rvert \lvert\dot{\psi}_s(x) \rvert \,d x + \int_{\R} \lvert Y_{\bar{\pi}_{\Delta}(s)}(x) \rvert \lvert\dot{\psi}_s(x) \rvert \,d x 
				\\
				& \le \lVert Y_{s}\rVert_{\alpha\beta} \lVert \dot{\psi}_s\rVert_{\frac{\alpha\beta}{\alpha\beta-1}} + \lVert Y_{\bar{\pi}_{\Delta}(s)}\rVert_{\alpha\beta} \lVert \dot{\psi}_s\rVert_{\frac{\alpha\beta}{\alpha\beta-1}}.
			\end{align*}		
			Therefore, a.s.
			\begin{align*}
				\int_0^t |\langle Y_{\bar{\pi}_{\Delta}(s)} - Y_s, \dot{\psi}_s \rangle| \,d s & \le (\sup_{s\le t} \lVert \dot{\psi}_s\rVert_{\frac{\alpha\beta}{\alpha\beta-1}}) \left[ \int_0^t \lVert Y_{s}\rVert_{\alpha\beta} \,d s + \int_0^t \lVert Y_{\bar{\pi}_{\Delta}(s)}\rVert_{\alpha\beta}  \, d s\right] 
				\\
				& = (\sup_{s\le t} \lVert \dot{\psi}_s\rVert_{\frac{\alpha\beta}{\alpha\beta-1}}) \left[ \left(\int_0^t \lVert Y_{s}\rVert_{\alpha\beta} \,d s\right)^{\frac{\alpha\beta}{\alpha\beta}} + \left(\int_0^t \lVert Y_{\bar{\pi}_{\Delta}(s)}\rVert_{\alpha\beta}  \, d s\right)^{\frac{\alpha\beta}{\alpha\beta}}\right]
				\\
				& \le (\sup_{s\le t} \lVert \dot{\psi}_s\rVert_{\frac{\alpha\beta}{\alpha\beta-1}})   \left[ \left(t^{\alpha\beta} \frac{1}{t}\int_0^t \lVert Y_{s}\rVert_{\alpha\beta}^{\alpha\beta} \,d s\right)^{\frac{1}{\alpha\beta}} + \left(t^{\alpha\beta}\frac{1}{t}\int_0^t \lVert Y_{\bar{\pi}_{\Delta}(s)}\rVert_{\alpha\beta}^{\alpha\beta}  \, d s\right)^{\frac{1}{\alpha\beta}}\right],   			
			\end{align*} using Jensen in the last line. The quantity above is finite by assumption (ii) and the fact that $Y \in L^{\alpha\beta}_{loc}(\R_+ \x \R)$. This implies that $s \mapsto |\langle Y_{\bar{\pi}_{\Delta}(s)} - Y_s, \dot{\psi}_s \rangle|$ is a.s. integrable on $[0,t]$.

			(b) Fix $s \in [0,t]$. Then 
			\begin{align*}
				|\langle Y_s, \frac{\partial^2}{\partial x^2}(\psi_s- \psi_{\pi_{\Delta}(s)}) \rangle| & \le \left(\sup_{x \in \R} |\frac{\partial^2}{\partial x^2} (\psi_s- \psi_{\pi_{\Delta}(s)})(x)|\right) \langle Y_s, 1\rangle, \text{ a.s.}
			\end{align*}
			We know that $Y_s$ is a finite measure, i.e. $\langle Y_s, 1\rangle< \infty$. Thus the RHS above converges to $0$ by our assumption (iii). 
			
			Let us introduce a new stopping time: for $k \in \N$, $$ \sigma_k = \inf\{s\ge 0 \mid \langle Y_s, 1 \rangle > k\}.$$ 
			For $s \le \sigma_k\wedge t$ we have, a.s.
			\begin{align*}
				|\langle Y_s, \frac{\partial^2}{\partial x^2}(\psi_s- \psi_{\pi_{\Delta}(s)}) \rangle| & \le \left(\sup_{x \in \R} |\frac{\partial^2}{\partial x^2} (\psi_s- \psi_{\pi_{\Delta}(s)})(x)| \right) \langle Y_s, 1\rangle\\
				& \le 2k \left(\sup_{\stackrel{s \le t}{x \in \R}} |\frac{\partial^2}{\partial x^2} \psi_s (x)|\right) < \infty
			\end{align*}
			by hypothesis. As $\sigma_k \to \infty$ as $k \to \infty$ the above is true for all $s \le t$. Thus we can apply DCT to obtain (b).
			
			(c)  Recall the notations introduced in the beginning of Section \ref{sec:prelim}. 	We have 
			\begin{align}\label{eq:stable_noise_PRM}
				L^{\alpha}(\,d x, \, d s) = \int_0^{\infty} z \tilde{N} (\,d x, \, d z,\, d s )	
			\end{align} 
			where $N(\,d x, \,d z,\,d s)$ is a PRM on $\R\x(0,\infty)^2$ with intensity $d x \, m_0(d z)\, d s$.
			
			Note that 
			\begin{align}\label{eq:time_dep_weak_form_2}
				& \int_0^t \int_{\R} Y_{s-}(x)^{\beta}(\psi_{\pi_{\Delta}(s)}(x) - \psi_s(x))L^{\alpha}(d x, \, d s) 
				\nonumber \\
				= & \int_0^t \int_0^1 \int_{\R} Y_{s-}(x)^{\beta}(\psi_{\pi_{\Delta}(s)}(x) - \psi_s(x))z \tilde{N}(d x , \, d z, \, d s)
				\nonumber \\
				& + \int_0^t \int_1^{\infty} \int_{\R} Y_{s-}(x)^{\beta}(\psi_{\pi_{\Delta}(s)}(x) - \psi_s(x))z \tilde{N}(d x, \, d z, \, d s).
			\end{align}
			Here we note that $\int_0^1 z^{\rho} m_0(dz) <\infty$ as $\rho \ge \alpha$ and $\int_1^{\infty} z^{\eta} m_0(dz) <\infty$ as $\eta < \alpha$.  Using the Burkholder-Davis-Gundy inequality for the first term, Fubini's theorem and Proposition \ref{prop:moment_est_Y_2} we have
			\begin{align}\label{eq:time_dep_weak_form_3}
				& \E \left( \left| \int_0^t \int_0^1 \int_{\R} Y_{s-}(x)^{\beta}(\psi_{\pi_{\Delta}(s)}(x) - \psi_s(x))z \tilde{N}(d x, \, d z, \, d s) \right|^{\rho} \right)
				\nonumber \\
				\le & C_{\rho}  \E \left( \left| \int_0^t \int_0^1 \int_{\R} Y_{s-}(x)^{2\beta}|\psi_{\pi_{\Delta}(s)}(x) - \psi_s(x)|^2 z^2 N(d x, \, d z, \, d s) \right|^{\rho/2} \right)
				\nonumber \\
				\le & C_{\rho} \E \left( \int_0^t \int_0^1 \int_{\R}  Y_{s-}(x)^{\rho \beta}|\psi_{\pi_{\Delta}(s)}(x) - \psi_s(x)|^{\rho}z^{\rho} N(d x, \, d z, \, d s)  \right)
				\nonumber \\
				\le & C_{\rho} \int_0^t \int_{\R} \left| \psi_{\pi_{\Delta}(s)}(x) - \psi_s(x) \right|^{\rho} \E( Y_{s-}(x)^{\rho\beta}) d x \, d s
				\nonumber \\
				\le & C_{\rho}  \left( \sup_{s} \| \psi_{\pi_{\Delta} (s)} - \psi_s\|_{\rho} \right)^{\rho} \left( \int_0^t  s^{-\rho\beta/2}\,d s\right)
				\nonumber \\
				= & C_{\rho}  \left( \sup_{s} \| \psi_{\pi_{\Delta} (s)} - \psi_s\|_{\rho} \right)^{\rho} t^{1-\rho\beta/2} \to 0
			\end{align}
			as $|\Delta| \to 0$. The second inequality again uses the fact about random sums described in \cite[Lemma 8.22]{pz07} as $\rho/2<1$. The last line follows from our assumption (i) of continuity of the map $s \mapsto \psi_s \in {\bf{L}}^{\rho}(\R)$, which implies uniform continuity of the same on $[0,t]$. For the second term in \eqref{eq:time_dep_weak_form_2} we proceed as in the previous calculation. Observe that $1<\eta \beta <\alpha$ by assumption and thus Proposition \ref{prop:moment_est_Y_2} is applicable in the following.
			\begin{align}\label{eq:time_dep_weak_form_4}
				& \E \left( \left| \int_0^t \int_1^{\infty} \int_{\R} Y_{s-}(x)^{\beta}(\psi_{\pi_{\Delta}(s)}(x) - \psi_s(x))z \tilde{N}(d x, \, d z, \, d s) \right|^{\eta} \right)
				\nonumber \\
				\le & \E \left( \left| \int_0^t \int_1^{\infty} \int_{\R} Y_{s-}(x)^{2\beta}(\psi_{\pi_{\Delta}(s)}(x) - \psi_s(x))^2z^2 N(d x, \, d z, \, d s) \right|^{\eta/2} \right)				
				\nonumber \\
				\le & \E  \int_0^t \int_1^{\infty} \int_{\R} \left| Y_{s-}(x)^{\eta\beta}(\psi_{\pi_{\Delta}(s)}(x) - \psi_s(x))^{\eta}z^{\eta} \right| N(\,d x, \, d z, \, d s) 
				\nonumber \\
				= &  \E  \int_0^t \int_1^{\infty} \int_{\R}  |Y_{s-}(x)|^{\eta\beta}\left|\psi_{\pi_{\Delta}(s)}(x) - \psi_s(x)\right|^{\eta}z^{\eta}  \, dx \, m_0 (dz) \, ds	
				\nonumber \\
				\le & C_{\eta} \int_0^t \int_{\R}  \E\left(|Y_{s-}(x)|^{\eta\beta}\right) \left|\psi_{\pi_{\Delta}(s)}(x) - \psi_s(x)\right|^{\eta} \, dx \, ds
				\nonumber \\
				\le &  C_{\eta} \int_0^t \int_{\R}  s^{\eta\beta/2} \left|\psi_{\pi_{\Delta}(s)}(x) - \psi_s(x)\right|^{\eta} \, dx \, ds
				\nonumber \\
				\le & C_{\eta} \left( \sup_{s} \| \psi_{\pi_{\Delta} (s)} - \psi_s\|_{\eta} \right)^{\eta} t^{1-\eta\beta/2}.
			\end{align}
		 By assumption (i) the RHS above converges to $0$ as $|\Delta| \to 0$. The calculations \eqref{eq:time_dep_weak_form_2}, \eqref{eq:time_dep_weak_form_3} and \eqref{eq:time_dep_weak_form_4} together prove (c). 
		\end{proof}
		
		In the last lemma we show how to turn the time dependent weak form of our SPDE \eqref{eq:time_dep_weak_form_0}, which was proved in the previous result, into a martingale. This will complete the proof of Proposition \ref{prop:time_dep_mart_prob}.
		
		\begin{lem}\label{lem:weak_form_to_mart}
			If \eqref{eq:time_dep_weak_form_0} holds for some smooth $\psi: [0, T] \x R \to \R$ and $0\le t \le T$, then 
			\begin{align}\label{eq:weak_form_to_mart_0}
				M^Y_t(\psi)  = e^{-\langle Y_t, \psi_t\rangle} - e^{-\langle Y_0, \psi_0\rangle}  - \int_0^t  e^{- \langle Y_{s-}, \psi_s\rangle } \left( -\langle Y_{s-}, \frac{1}{2} \Delta \psi_s + \frac{\partial}{\partial s}\psi_s  \rangle + \langle Y_{s-}^{\alpha \beta}, \psi_s^{\alpha} \rangle\right)  \,d s
			\end{align}
			is an $\mc{F}^Y$-martingale.
		\end{lem}
		
		\begin{proof}
			The proof is an application of Ito's formula. Using the representation \eqref{eq:stable_noise_PRM} and some algebraic manipulations we have,
			\begin{align}\label{eq:weak_form_to_mart_1}
				\langle Y_t, \psi\rangle  = &  \langle  Y_0, \psi_0  \rangle  + \int_0^t \langle  Y_s, \frac{1}{2}\Delta \psi_s  + \frac{\partial}{\partial s}\psi_s \rangle \,d s
				+  \int_0^t \int_0^{\infty} \int_{\R} Y_{s-}(x)^{\beta} \psi_s(x) z \tilde{N} (\,d x, \, d z,\, d s )
				\nonumber \\
				= &  \langle  Y_0, \psi_0  \rangle  + \int_0^t \langle  Y_s, \frac{1}{2}\Delta \psi_s + \frac{\partial}{\partial s}\psi_s  \rangle \,d s
				+  \int_0^t \int_0^{1} \int_{\R} Y_{s-}(x)^{\beta} \psi_s(x) z \tilde{N} (\,d x, \, d z,\, d s)
				\nonumber \\
				& + \int_0^t \int_1^{\infty} \int_{\R} Y_{s-}(x)^{\beta} \psi_s(x) z N (\,d x, \, d z,\, d s ) -  \int_0^t \int_1^{\infty} \int_{\R} Y_{s-}(x)^{\beta} \psi_s(x) z \, d x \, m_0(d z)\, d s
				\nonumber \\
				= & \langle  Y_0, \psi_0  \rangle  + \int_0^t  \left[ \langle  Y_s, \frac{1}{2}\Delta \psi_s + \frac{\partial}{\partial s}\psi_s  \rangle  - \int_1^{\infty} z \langle Y_{s-}(\cdot)^{\beta} , \psi_s \rangle \, m_0(dz)\right] \,d s
				\nonumber \\
				& + \int_0^t \int_0^{1} \int_{\R} Y_{s-}(x)^{\beta} \psi_s(x) z \tilde{N} (\,d x, \, d z,\, d t )
				+ \int_0^t \int_1^{\infty} \int_{\R} Y_{s-}(x)^{\beta} \psi_s(x) z N (\,d x, \, d z,\, d t)	
			\end{align}		
			Since $\int_1^{\infty} z \, m_0(dz) = \frac{\alpha}{\Gamma(2-\alpha)}$ the above can be written formally as 
			\begin{align}\label{eq:weak_form_to_mart_2}
				d \langle Y_t, \psi \rangle =&  \left( \langle Y_t, \frac{\Delta}{2}\psi_t + \frac{\partial}{\partial t}\psi_t \rangle - \frac{\alpha}{\Gamma(2-\alpha)} \langle Y_{t-}^{\beta}, \psi_t \rangle\right) \,d t
				\nonumber \\
				& +  \int_0^{1} \int_{\R} Y_{t-}(x)^{\beta} \psi_t(x) z \tilde{N} (\,d t, \,d z,\,d x )
				+ \int_1^{\infty} \int_{\R} Y_{t-}(x)^{\beta} \psi_t(x) z N (\,d t, \,d z,\,d x).
			\end{align}
			
			Ito's formula as given in \cite[Theorem 4.4.7]{app09} can now be applied with $f(x)= e^{-x}$ and 
			\begin{align*}
				G(t) =  \left( \langle Y_t, \frac{\Delta}{2}\psi_t + \frac{\partial}{\partial t}\psi_t \rangle - \frac{\alpha}{\Gamma(2-\alpha)} \langle Y_{t-}^{\beta}, \psi_t \rangle\right).
			\end{align*} 
			We have from \eqref{eq:weak_form_to_mart_2},
			\begin{align}\label{eq:weak_form_to_mart_3}
				& e^{-\langle Y_t, \psi_t  \rangle} - e^{-\langle Y_0, \psi_0\rangle} = f(\langle Y_t, \psi_t  \rangle) - f(\langle Y_0, \psi_0\rangle) 
				\nonumber \\
				= & \int_0^t f'(\langle Y_{s-}, \psi_s\rangle) G(s) \,d s
				+ \int_0^t \int_1^{\infty} \int_{\R} \left[ f(\langle Y_{s-}, \psi_s\rangle + Y_{s-}(x)^{\beta} \psi_s(x) z ) - f(\langle Y_{s-}, \psi_s\rangle) \right] N (\,d x, \, d z,\, d s)
				\nonumber \\
				& + \int_0^t \int_0^1 \int_{\R} \left[ f(\langle Y_{s-}, \psi_s\rangle + Y_{s-}(x)^{\beta} \psi_s(x) z ) - f(\langle Y_{s-}, \psi_s\rangle) \right] \tilde{N}	(\,d x, \, d z,\, d s)
				\nonumber \\
				& + \int_0^t \int_0^1 \int_{\R} \left[ f(\langle Y_{s-}, \psi_s\rangle + Y_{s-}(x)^{\beta} \psi_s(x) z ) - f(\langle Y_{s-}, \psi_s\rangle)-   Y_{s-}(x)^{\beta} \psi_s(x) z f'(\langle Y_{s-}, \psi_s\rangle) \right] \,d s \,m_0(d z)\,d x
				\nonumber \\
				= & - \int_0^t e^{-\langle Y_{s-}, \psi_s\rangle} \left( \langle Y_{s-}, \frac{\Delta}{2}\psi_s +\frac{\partial}{\partial s}\psi_s \rangle - \frac{\alpha}{\Gamma(2-\alpha)} \langle Y_{s-}^{\beta}, \psi_s \rangle\right) \, d s
				\nonumber \\
				& + \int_0^t \int_1^{\infty} \int_{\R} e^{-\langle Y_{s-}, \psi_s\rangle} \left( e^{-Y_{s-}(x)^{\beta} \psi_s(x) z} -1\right) N (\,d x, \, d z,\, d s)
				\nonumber \\
				& + \int_0^t \int_0^1 \int_{\R} e^{-\langle Y_{s-}, \psi_s\rangle} \left( e^{-Y_{s-}(x)^{\beta} \psi_s(x) z} -1 \right) \tilde{N}	(\,d x, \, d z,\, d s)
				\nonumber \\
				& + \int_0^t \int_0^1 \int_{\R} e^{-\langle Y_{s-}, \psi_s\rangle} \left( e^{-Y_{s-}(x)^{\beta} \psi_s(x) z} -1 +   Y_{s-}(x)^{\beta} \psi_s(x) z \right) \, d x \, m_0(d z)\, d s
				\nonumber \\
				= & - \int_0^t e^{-\langle Y_{s-}, \psi_s\rangle} \left( \langle Y_{s-}, \frac{\Delta}{2}\psi_s + \frac{\partial}{\partial s}\psi_s \rangle  - \frac{\alpha}{\Gamma(2-\alpha)} \langle Y_{s-}^{\beta}, \psi_s \rangle\right) \, d s
				\nonumber \\
				& + \left[ \int_0^t \int_1^{\infty} \int_{\R} e^{-\langle Y_{s-}, \psi_s\rangle} \left( e^{-Y_{s-}(x)^{\beta} \psi_s(x) z} -1\right) \tilde{N} (\,d x, \, d z,\, d s) \right.
				\nonumber \\
				& \left. + \int_0^t \int_0^1 \int_{\R} e^{-\langle Y_{s-}, \psi_s\rangle} \left( e^{-Y_{s-}(x)^{\beta} \psi_s(x) z} -1 \right) \tilde{N}	(\,d x, \, d z,\, d s)\right]
				\nonumber \\
				& + \int_0^t \int_1^{\infty} \int_{\R} e^{-\langle Y_{s-}, \psi_s\rangle} \left( e^{-Y_{s-}(x)^{\beta} \psi_s(x) z} -1 \right) \, d x \, m_0(d z)\, d s 
				\nonumber \\
				& +  \int_0^t \int_0^{1} \int_{\R} e^{-\langle Y_{s-}, \psi_s\rangle} \left( e^{-Y_{s-}(x)^{\beta} \psi_s(x) z} -1 + Y_{s-}(x)^{\beta} \psi_s(x) z \right) \, d x \, m_0(d z) \, d s 			
			\end{align}
			adding and subtracting the term $ \int_0^t \int_1^{\infty} \int_{\R} e^{-\langle Y_{s-}, \psi_s\rangle} \left( e^{-Y_{s-}(x)^{\beta} \psi_s(x) z} -1 \right) \,d s \,m_0(d z)\,d x$ in the last line. Note that the term in the square bracket above is 
			\begin{align*}
				M_t(\psi) = \int_0^t \int_0^{\infty} \int_{\R} e^{-\langle Y_{s-}, \psi_s\rangle} \left( e^{-Y_{s-}(x)^{\beta} \psi_s(x) z} -1\right) \tilde{N} (\,d x, \, d z,\, d s) 
			\end{align*} and it is an $\mc{F}^Y$-martingale. We consider the last term in the RHS of \eqref{eq:weak_form_to_mart_3}. Recall the definition of $m_0$ from \eqref{eq:m_0} and the fact that for $y \ge 0$, 
			\begin{align*}
				\frac{\alpha(\alpha -1)}{\Gamma(2-\alpha)} \int_{0+}^{\infty} \left( e^{-\lambda y} -1 + \lambda y\right)\lambda ^{-\alpha -1} \,d \lambda = y^{\alpha}.
			\end{align*}
			From these we can get,
			\begin{align}\label{eq:weak_form_to_mart_4}
				& \int_0^t \int_0^{1} \int_{\R} e^{-\langle Y_{s-}, \psi_s\rangle} \left( e^{-Y_{s-}(x)^{\beta} \psi_s(x) z} -1 + Y_{s-}(x)^{\beta} \psi_s(x) z \right) \, d x \, m_0(d z) \, d s 
				\nonumber \\
				= & 	\int_0^t  \int_{\R} e^{-\langle Y_{s-}, \psi_s\rangle} \left[\frac{\alpha(\alpha -1)}{\Gamma(2-\alpha)}\int_0^{\infty}  \left( e^{-Y_{s-}(x)^{\beta} \psi_s(x) z} -1 + Y_{s-}(x)^{\beta} \psi_s(x) z \right)  z^{-1-\alpha}\,d z \right] \, d x \, d s 	
				\nonumber \\
				& - \int_0^t \int_1^{\infty} \int_{\R} e^{-\langle Y_{s-}, \psi_s\rangle} \left( e^{-Y_{s-}(x)^{\beta} \psi_s(x) z} -1 + Y_{s-}(x)^{\beta} \psi_s(x) z \right) \,d s \,m_0(d z)\,d s
				\nonumber \\ 
				= & \int_0^t  \int_{\R}  e^{-\langle Y_{s-}, \psi_s\rangle} Y_{s-}(x)^{\alpha\beta} \psi_s(x)^{\alpha}  \, dx \, d s 
				\nonumber \\
				& -  \int_0^t \int_1^{\infty} \int_{\R} e^{-\langle Y_{s-}, \psi_s\rangle} \left( e^{-Y_{s-}(x)^{\beta} \psi_s(x) z} -1 + Y_{s-}(x)^{\beta} \psi_s(x) z \right) \, d x \, m_0(d z)\, d s.
			\end{align}
			To finish the proof use the result of \eqref{eq:weak_form_to_mart_4} in \eqref{eq:weak_form_to_mart_3}. By algebraic manipulations we have,
			\begin{align}
				& e^{-\langle Y_t, \psi_t  \rangle} - e^{-\langle Y_0, \psi_0\rangle}
				\nonumber \\
				= & - \int_0^t e^{-\langle Y_{s-}, \psi_s\rangle} \left( \langle Y_{s-}, \frac{\Delta}{2}\psi_s + \frac{\partial}{\partial s}\psi_s \rangle - \frac{\alpha}{\Gamma(2-\alpha)} \langle Y_{s-}^{\beta}, \psi_s \rangle\right) \, d s + M_t(\psi)
				\nonumber \\
				& + \int_0^t \int_1^{\infty} \int_{\R} e^{-\langle Y_{s-}, \psi_s\rangle} \left( e^{-Y_{s-}(x)^{\beta} \psi_s(x) z} -1 \right) \,d s \,m_0(d z)\,d x 
				+ \int_0^t   e^{-\langle Y_{s-}, \psi_s\rangle} \langle Y_{s-}^{\alpha\beta} ,\psi_s^{\alpha} \rangle \,d s
				\nonumber \\
				& -  \int_0^t \int_1^{\infty} \int_{\R} e^{-\langle Y_{s-}, \psi_s\rangle} \left( e^{-Y_{s-}(x)^{\beta} \psi_s(x) z} -1 + Y_{s-}(x)^{\beta} \psi_s(x) z \right) \, d x \, m_0(d z)\, d s,
				\nonumber \\
				= & - \int_0^t e^{-\langle Y_{s-}, \psi_s\rangle} \left( \langle Y_{s-}, \frac{\Delta}{2}\psi_s + \frac{\partial}{\partial s}\psi_s \rangle - \frac{\alpha}{\Gamma(2-\alpha)} \langle Y_{s-}^{\beta}, \psi_s \rangle + \langle Y_{s-}^{\alpha\beta} ,\psi_s^{\alpha} \rangle \right) \, d s + M_t(\psi)
				\nonumber \\
				& -  \int_0^t \int_1^{\infty} \int_{\R} e^{-\langle Y_{s-}, \psi_s\rangle}  Y_{s-}(x)^{\beta} \psi(x) z  \, d x \, m_0(d z)\, d s,
				\nonumber \\
				= & - \int_0^t e^{-\langle Y_{s-}, \psi_s\rangle} \left( \langle Y_{s-}, \frac{\Delta}{2}\psi_s + \frac{\partial}{\partial s}\psi_s \rangle + \langle Y_{s-}^{\alpha\beta} ,\psi_s^{\alpha} \rangle \right) \, d s + M_t(\psi)
			\end{align}
			again using the fact that $\int_1^{\infty} z \, m_0(dz) = \frac{\alpha}{\Gamma(2-\alpha)}$.
		\end{proof}

	\bibliographystyle{alpha} 
	\bibliography{ref}       
	
	
\end{document}